\documentclass[letterpaper,12pt]{amsart}

\usepackage[hmargin=2.0cm,vmargin=2.0cm]{geometry}

\usepackage{amscd, amssymb}
\usepackage{amsmath,amscd}
\usepackage[driverfallback=hypertex]{hyperref}
\usepackage{nameref,zref-xr}                    %to include reference to other papers
\zxrsetup{toltxlabel}
\zexternaldocument*[calmodulin-]{real_Moduli_20_11}
\usepackage{comment}
\usepackage{graphicx}
\usepackage[usenames,dvipsnames]{color}
\usepackage{bm}
\usepackage{enumitem}
\usepackage[all]{xy}   %For commutative diagrams
\setlist[enumerate,1]{label={(\alph*)}}
\setlist[enumerate,2]{label={(\roman*)}}
\input xy
\xyoption{all}
\usepackage{epstopdf}

\usepackage{tikz}

\newtheorem{thm}{Theorem}[section]
\newtheorem{prop}[thm]{Proposition}

\newtheorem{lemma}[thm]{Lemma}
\newtheorem{cor}[thm]{Corollary}

\newtheorem{conj}{Conjecture}

\newcommand{\p}{\partial}

\theoremstyle{definition}
\newtheorem{definition}[thm]{Definition}
\newtheorem{nn}[thm]{Notation}

\theoremstyle{remark}
\newtheorem{rmk}[thm]{Remark}

\newcommand{\R}{\mathbb R}

\newcommand{\Z}{\mathbb Z}
\newcommand{\CM}{{\mathcal{M}}}
\newcommand{\oCM}{{\overline{\mathcal{M}}}}

\newcommand{\CL}{{\mathbb{L}}}

\newcommand{\CS}{{\mathcal{S}}}

\newcommand{\blangle}{{\big\langle}}
\newcommand{\bblangle}{{\big\langle}{\big\langle}}
\newcommand{\brangle}{{\big\rangle}}
\newcommand{\bbrangle}{{\big\rangle}{\big\rangle}}

\newcommand{\tr}{\rho}

\newcommand{\mcH}{\mathcal H}
\newcommand{\mcA}{\mathcal A}
\newcommand{\mbR}{\mathbb R}
\newcommand{\mbZ}{\mathbb Z}
\newcommand{\half}{\frac{1}{2}}

\renewcommand{\>}{\right >}
\newcommand{\mbC}{\mathbb C}
\newcommand{\mcR}{\mathcal R}
\newcommand{\tlambda}{\widetilde\lambda}
\newcommand{\mcL}{\mathcal L}
\newcommand{\mcU}{\mathcal U}
\newcommand{\tF}{\widetilde F}
\newcommand{\tf}{\widetilde f}
\newcommand{\ttau}{\widetilde\tau}

\newcommand{\tc}{\widetilde c}

\newcommand{\diag}{\mathop{\mathrm{diag}}\nolimits}
\renewcommand{\tr}{\mathop{\mathrm{tr}}\nolimits}
\renewcommand{\deg}{\mathop{\mathrm{deg}}\nolimits}
\newcommand{\Faces}{\mathop{\mathrm{Faces}}\nolimits}
\newcommand{\Edges}{\mathop{\mathrm{Edges}}\nolimits}

\newcommand{\Aut}{\mathop{\mathrm{Aut}}\nolimits}
\renewcommand{\Re}{\mathop{\mathrm{Re}}\nolimits}
\renewcommand{\Im}{\mathop{\mathrm{Im}}\nolimits}

\renewcommand{\d}{\partial}

\numberwithin{equation}{section}

\begin{document}

\title[Proof of the open analog of Witten's conjecture]{Matrix models and a proof of the open analog of Witten's conjecture}
\author{Alexandr Buryak}
\address{A.~Buryak:\newline
Department of Mathematics, ETH Zurich,\newline
HG G 27.1, R\"amistrasse 101 8092, Z\"urich, Switzerland}
\email{buryaksh@gmail.com}

\author{Ran J. Tessler}
\address{R.~J.~Tessler:\newline
ETH Inst. f\"ur Theoretische Studien, ETH Zurich,\newline
CLV A4, Clausiusstrasse 47 8092 Z\"urich, Switzerland}
\email{ran.tessler@mail.huji.ac.il}


%\date{}

\begin{abstract}
In a recent work, R.~Pandharipande, J.~P.~Solomon and the second author have initiated a study of the intersection theory on the moduli space of Riemann surfaces with boundary. They conjectured that the generating series of the intersection numbers satisfies the open KdV equations. In this paper we prove this conjecture. Our proof goes through a matrix model and is based on a Kontsevich type combinatorial formula for the intersection numbers that was found by the second author.
\end{abstract}

\maketitle

\tableofcontents

\section{Introduction}

The study of the intersection theory on the moduli space of Riemann surfaces with boundary (often viewed, with the boundary removed, as open Riemann surfaces) was recently initiated in~\cite{PST14}. The authors constructed a descendent theory in genus~$0$ and obtained a complete description of it. In all genera, they conjectured that the generating series of the descendent integrals satisfies the open KdV equations. This conjecture can be considered as an open analog of the famous Witten's conjecture~\cite{Wit91}. 
The construction of the higher genus moduli and intersection theory was found by J. Solomon and R.T. in~\cite{JSRTa}. The details of these constructions also appear in~\cite{RT}, Section~2. A combinatorial formula for the open intersection numbers in all genera was found in~\cite{RT}.

In this paper, using the combinatorial formula from~\cite{RT}, we present a matrix integral for the generating series of the open intersection numbers. Then applying some analytical tools to this matrix integral, we prove the main conjecture from~\cite{PST14}.

The introduction is organized as follows. In Section~\ref{wittc} we briefly recall the original conjecture of E.~Witten from \cite{Wit91}. In Section~\ref{konts} we recall Kontsevich's combinatorial formula and Kontsevich's proof (\cite{Kon92}) of Witten's conjecture. Section~\ref{okdv} contains a short account of the main constructions and conjectures in the open intersection theory from \cite{PST14,JSRTa,Bur14a,Bur14b,JSRTb}. Section~\ref{ocomb} describes the combinatorial formula of \cite{RT} for the open intersection numbers of \cite{PST14,JSRTa}.

\subsection{Witten's conjecture}\label{wittc}

\begin{nn}
Throughout this text $[n]$ will denote the set $\{1,2,\ldots,n\}.$
\end{nn}

\subsubsection{Intersection numbers}

A compact Riemann surface is a compact connected smooth complex curve. Given a fixed genus~$g$ and a non-negative integer $l,$ the moduli space of all compact Riemann surfaces of genus
$g$ with $l$ marked points is denoted by $\CM_{g,l}.$ P.~Deligne and D.~Mumford defined a natural compactification of it via stable curves in~\cite{DM69} in 1969.
Given $g,l$ as above, a stable curve is a compact connected complex curve with $l$ marked points and finitely many singularities, all of which are simple nodes. The collection of
marked points and nodes is the set of special points of the curve. We require that all the special points are distinct and that the automorphism group of the curve is finite. The
moduli of stable marked curves of genus $g$ with $l$ marked points is denoted by~$\oCM_{g,l}$ and is a compactification of~$\CM_{g,l}.$ It is known that this space is a non-singular
complex orbifold of complex dimension $3g-3+l.$ For the basic theory the reader is referred to~\cite{DM69,HM98}.

In his seminal paper~\cite{Wit91}, E.~Witten, motivated by theories of $2$-dimensional quantum gravity, initiated new directions in the study of $\oCM_{g,l}$. For each marking index $i$
he considered the tautological line bundles
$$\CL_i \rightarrow \oCM_{g,l}$$
whose fiber over a point
$$[\Sigma,z_1,\ldots,z_l]\in \oCM_{g,l}$$
is the complex cotangent space $T_{z_i}^*\Sigma$ of $\Sigma$ at $z_i$. Let
$$\psi_i\in H^2(\oCM_{g,l};\mathbb{Q})$$
denote the first Chern class of $\CL_i$, and write
\begin{equation}\label{products}
\blangle\tau_{a_1} \tau_{a_2} \cdots \tau_{a_l}\brangle_g^c:=\int
_{\overline {M}_{g,l}} \psi_1^{a_1} \psi_2^{a_2} \cdots \psi_l^{a_l}.
\end{equation}
The integral on the right-hand side of~\eqref{products} is well-defined, when the stability condition
$$2g-2+l >0$$
is satisfied, all the $a_i$'s are non-negative integers, and the dimension constraint
\begin{equation*}
3g-3+l=\sum_i a_i
\end{equation*}
holds. In all other cases $\blangle\prod_{i=1}^{l} \tau_{a_i}\brangle_g^c$ is defined to be zero.
The intersection products \eqref{products} are often called {\em descendent integrals} or {\em intersection numbers}. Note that the genus is uniquely determined by the exponents~$\{a_i\}$.

Let $t_i$ (for $i\geq 0$) and $u$ be formal variables, and put
$$\gamma:=\sum_{i=0}^{\infty} t_i \tau_i.$$
Let
$$F^c_g(t_0, t_1,\ldots):=\sum_{n=0}^{\infty} \frac{\blangle\gamma^n\brangle^c_g}{n!}$$
be the generating function of the genus $g$ descendent integrals \eqref{products}. The bracket $\blangle\gamma^n\brangle^c_g$ is defined by the monomial expansion and the
multilinearity in the variables $t_i$. Concretely,
$$F_g^c(t_0, t_1, ...)= \sum_{\{n_i\}}\blangle\tau_0^{n_0} \tau_1^{n_1} \tau_2^{n_2} \cdots\brangle_g^c\prod_{i=0}^{\infty}\frac{t_i^{n_i}}{n_i!},$$
where the sum is over all sequences of non-negative integers $\{n_i\}$ with finitely many non-zero terms.
The generating series
\begin{equation}\label{v34}
F^c:=\sum_{g=0}^{\infty} u^{2g-2} F_g^c
\end{equation}
is called the {\it (closed) free energy}. The exponent $\tau^c:=\exp(F^c)$ is called the {\it (closed) partition function}.

\subsubsection{KdV equations}

Put $\blangle\blangle\tau_{a_1} \tau_{a_2} \cdots \tau_{a_l}\brangle\brangle^c:=\frac{\p^l F^c}{\p t_{a_1}\p t_{a_2}\cdots\p t_{a_l}}$. Witten's conjecture~(\cite{Wit91}) says that the closed partition function~$\tau^c$ becomes a tau-function of the KdV hierarchy after the change of variables~$t_n=(2n+1)!!T_{2n+1}$. In particular, it implies that the closed free energy~$F^c$ satisfies the following system of partial differential equations:
\begin{gather}\label{eq:kdv equations}
(2n+1)u^{-2} \blangle\blangle\tau_n \tau_0^2 \brangle\brangle^c = \blangle\blangle\tau_{n-1} \tau_0\brangle\brangle^c\blangle\blangle\tau_0^3\brangle\brangle^c +
2\blangle\blangle\tau_{n-1}\tau_0^2\brangle\brangle^c\blangle\blangle\tau_0^2\brangle\brangle^c+
\frac{1}{4}\blangle\blangle\tau_{n-1} \tau_0^4\brangle\brangle^c,\quad n\ge 1.
\end{gather}
These equations are known in mathematical physics as the KdV equations. E.~Witten~(\cite{Wit91}) proved that the intersection numbers~\eqref{products} satisfy the string equation
\begin{gather*}
\left\langle\tau_0 \prod_{i=1}^{l} \tau_{a_i}\right\rangle^c_g =
\sum_{j=1}^{l} \left\langle\tau_{a_j-1} \prod_{i\neq j} \tau_{a_i}\right\rangle^c_g,
\end{gather*}
for $2g-2+l>0$. This equation can be rewritten as the following differential equation:
\begin{gather}\label{eq:string}
\frac{\d F^c}{\d t_0}=\sum_{i\ge 0}t_{i+1}\frac{\d F^c}{\d t_i}+\frac{t_0^2}{2 u^2}.
\end{gather}
E.~Witten also showed that the KdV equations~\eqref{eq:kdv equations} together with the string equation~\eqref{eq:string} actually determine the closed free energy $F^c$ completely.

\subsubsection{Virasoro equations}\label{subsubsection:closed Virasoro}

There was a later reformulation of Witten's conjecture due to R.~Dijkgraaf, E.~Verlinde and H.~Verlinde (\cite{DVV91}) in terms of the Virasoro algebra. Define differential operators
$L_n$, $n\ge -1$, by
\begin{align}
\label{lminus}
L_{-1}:=&-\frac{\p}{\p t_0}+\sum_{i=0}^{\infty} t_{i+1}\frac{\p}{\p t_i}+\frac{u^{-2}}{2} t_0^2,\\
\nonumber
%\label{lzero}
L_0:=&-\frac{3}{2} \frac{\p}{\p t_1}+ \sum_{i=0}^{\infty}
\frac{2i+1}{2} t_i \frac{\p}{\p t_i} + \frac{1}{16},
\end{align}
while for $n\geq 1,$
\begin{align}
%\label{repl}
L_n:=&-\frac{3\cdot 5\cdot 7 \cdots (2n+3)}{2^{n+1}} \frac{\p}{\p t_{n+1}}+\sum_{i=0}^{\infty} \frac{(2i+1)(2i+3) \cdots (2i+2n+1)}{2^{n+1}} t_i \frac{\p}{\p t_{i+n}}\notag\\
& + \frac{u^2}{2} \sum_{i=0}^{n-1} (-1)^{i+1} \frac{(-2i-1)(-2i+1) \cdots
(-2i+2n-1)}{2^{n+1}} \frac{\p^2}{\p t_i \p t_{n-1-i}}.\notag
\end{align}
The Virasoro equations say that the operators $L_n$, $n\geq-1$, annihilate the closed partition function~$\tau^c$:
\begin{equation}
\label{vira}
L_n\tau^c=0,\quad n\ge -1.
\end{equation}
It is easy to see that the Virasoro equations completely determine all intersection numbers. R.~Dijkgraaf, E.~Verlinde and H.~Verlinde (\cite{DVV91}) proved that this description is
equivalent to the one given by the KdV equations and the string equation.

Witten's conjecture was proven by M.~Kontsevich~\cite{Kon92}. See~\cite{KL07,Mir07,OP05} for other proofs.

\subsection{Kontsevich's Proof}\label{konts}

Kontsevich's proof~\cite{Kon92} of Witten's conjecture consisted of two parts.
The first part was to prove a combinatorial formula for the gravitational descendents. Let~$G_{g,n}$ be the set of isomorphism classes of trivalent ribbon graphs of genus~$g$ with~$n$ faces and together with a numbering $\Faces(G)\simeq [n]$. Denote by~$V(G)$ the set of vertices of a graph $G\in G_{g,n}$. Let us
introduce formal variables $\lambda_i$,~$i\in [n]$. For an edge $e\in\Edges(G),$ let $\lambda(e):=\frac{1}{\lambda_i+\lambda_j},$ where~$i$ and~$j$ are the numbers of faces adjacent
to~$e$. Then we have
\begin{gather}\label{eq:Kontsevich's formula}
\sum_{a_1,\ldots,a_n\ge 0}\left<\prod_{i=1}^{n} \tau_{a_i}\right>_g^c\prod_{i=1}^n\frac{(2a_i-1)!!}{\lambda_i^{2a_i+1}} = \sum_{G\in
G_{g,n}}\frac{2^{|\Edges(G)|-|V(G)|}}{|\Aut(G)|}\prod_{e\in\Edges(G)}\lambda(e).
\end{gather}

The second step of Kontsevich's proof was to translate the combinatorial formula into a matrix integral.
Then, by using non-trivial analytical tools and the theory of tau-functions of the KdV hierarchy, he was able to prove that $\tau^c$ is a tau-function of the KdV hierarchy and, hence, the free energy~$F^c$ satisfies the KdV equations~\eqref{eq:kdv equations}.

\subsection{Open intersection numbers and the open KdV equations}\label{okdv}

\subsubsection{Open intersection numbers}

In \cite{PST14} R. Pandharipande, J. Solomon and R.T. constructed an intersection theory on the moduli space of stable marked disks. Let~$\oCM_{0,k,l}$ be the moduli space of stable marked
disks with~$k$ boundary marked points and~$l$ internal marked points. This space carries a natural structure of a compact smooth oriented manifold with corners. One can easily define
the tautological line bundles $\CL_i,$ for $i\in[l]$, as in the closed case.

In order to define gravitational descendents, as in~\eqref{products}, we must specify boundary conditions. Indeed, given a smooth compact connected oriented orbifold with boundary, $(M,\partial M)$ of dimension~$n$, the Poincar\'e-Lefschetz duality shows that
\[
H^n(M,\partial M;\mathbb{Q})\cong H_0(M;\mathbb{Q})\cong\mathbb{Q}.
\]
Thus, given a vector bundle on a manifold with boundary, only \emph{relative Euler
class}, relative to nowhere vanishing boundary conditions, can be integrated to give a number. The main construction in~\cite{PST14} is a construction of boundary conditions for $\CL_i\to\oCM_{0,k,l}.$ In~\cite{PST14}, vector spaces $\CS_i = \CS_{i,0,k,l}$ of
\emph{multisections} of $\CL_i\to\partial\oCM_{0,k,l},$ which satisfy the following requirements, were defined. Suppose $a_1,\ldots,a_l$ are non-negative integers with $2\sum_i
a_i=\dim_\R\oCM_{0,k,l}=k+2l-3,$ then
\begin{enumerate}[label={(\alph*)}]
\item\label{it:canonic1}
For any generic choice of multisections $s_{ij}\in\CS_i,$ for $1\leq j\leq a_i,$ the multisection
$$
s=\bigoplus_{\substack{i\in [l]\\1\le j\le a_i}} s_{ij}
$$
vanishes nowhere on $\partial\oCM_{0,k,l}.$
\item\label{it:canonic2}
For any two such choices $s$ and $s'$ we have
$$\int_{\oCM_{0,k,l}}e(E,s) = \int_{\oCM_{0,k,l}}e(E,s'),$$
where $E:=\bigoplus_i\CL_i^{a_i},$ and $e(E,s)$ is the relative Euler class.
\end{enumerate}
The multisections $s_{ij}$, as above, are called \emph{canonical}. With this construction the open gravitational descendents in genus~$0$ are defined by
\begin{equation}
\label{openproducts}
\blangle\tau_{a_1} \tau_{a_2} \cdots \tau_{a_l}\sigma^k\brangle_0^o:=2^{-\frac{k-1}{2}}\int_{\oCM_{0,k,l}}e(E,s),
\end{equation}
where $E$ is as above and $s$ is canonical.

In a forthcoming paper \cite{JSRTa}, J.~Solomon and R.T. define a generalization for all genera. In \cite{JSRTa} a moduli space $\oCM_{g,k,l}$ which classifies stable genus $g$ Riemann surfaces with boundary, together with some additional structure, is constructed. By the genus of a surface with boundary we mean the genus of the doubled surface.
The moduli space $\oCM_{g,k,l}$ is a smooth oriented compact orbifold with corners, of real dimension
\begin{gather}\label{open dimension}
3g-3+k+2l.
\end{gather}
The stability condition is $2g-2+k+2l>0.$
Note that naively, without adding an extra structure, the moduli of stable surfaces with boundary is non-orientable for $g>0.$
The construction of the moduli space, the definition of open descendents and an alternative proof of orientability also appear in~\cite{RT}, Section~2.

On $\oCM_{g,k,l}$ one defines vector spaces $\CS_i=\CS_{i,g,k,l}$, for
$i\in[l],$ for which the genus $g$ analogs of requirements \ref{it:canonic1},\ref{it:canonic2} from above hold.
Write
\begin{equation}
\label{openproducts}
\blangle\tau_{a_1} \tau_{a_2} \cdots \tau_{a_l}\sigma^k\brangle_g^o:=2^{-\frac{g+k-1}{2}}\int_{\oCM_{g,k,l}}e(E,s),
\end{equation}
for the corresponding higher genus descendents. Introduce one more formal variable $s$. The \emph{open free energy} is the generating function
\begin{equation}\label{eq:gen_func}
F^o(s,t_0,t_1,\ldots;u) := \sum_{g=0}^\infty u^{g-1} \sum_{n=0}^\infty
\frac{\blangle {\gamma^n \delta^k}
\brangle_g^o}{n!k!},
\end{equation}
where $\gamma:=\sum_{i\ge 0} t_i\tau_i$, $\delta:=s\sigma$, and again we use the monomial expansion and the multilinearity in the variables $t_i,s.$

\subsubsection{Open KdV and open Virasoro equations}

The following initial condition follows easily from the definitions (\cite{PST14}):
\begin{gather}\label{eq:open initial condition}
\left.F^o\right|_{t_{\ge 1}=0}=u^{-1}\frac{s^3}{6}+u^{-1}t_0 s.
\end{gather}
In~\cite{PST14} the authors conjectured the following equations:
\begin{align}\label{gvvt}
\frac{\partial F^o}{\partial t_0} =& \sum_{i=0}^{\infty}t_{i+1}\frac{\partial F^o}{\partial t_i} + u^{-1}s,\\
\frac{\partial F^o}{\partial t_1} =& \sum_{i=0}^{\infty}\frac{2i+1}{3}t_{i}\frac{\partial F^o}{\partial t_i} + \frac{2}{3}s\frac{\partial F^o}{\partial s}+\frac{1}{2}.\label{eq:open dilaton}
\end{align}
They were called the open string and the open dilaton equation correspondingly.

Put
$\blangle\blangle\tau_{a_1} \tau_{a_2} \cdots \tau_{a_l} \sigma^k
\brangle\brangle^o :=
\frac{\p^{l+k} F^o}{\p t_{a_1}\p t_{a_2}\cdots\p t_{a_l}\p s^k}.$ The main conjectures in \cite{PST14} are

\begin{conj}[Open analog of Witten's conjecture]\label{conjecture:open KdV}
The following system of equations is satisfied:
\begin{multline}\label{eq:openkdv}
(2n+1)u^{-1}\bblangle \tau_n \bbrangle^o=u\bblangle
\tau_{n-1} \tau_0\bbrangle^c  \bblangle \tau_0\bbrangle^o -\frac{u}{2} \bblangle \tau_{n-1}\tau_0^2 \bbrangle^c+\\
 + 2    \bblangle \tau_{n-1}\bbrangle^o
\bblangle\sigma\bbrangle^o + 2 \bblangle \tau_{n-1} \sigma \bbrangle^o,\quad n\ge 1.
\end{multline}
\end{conj}
\noindent In~\cite{PST14} equations~\eqref{eq:openkdv} were called the open KdV equations. It is easy to see that~$F^o$ is fully determined by the open KdV equations~\eqref{eq:openkdv}, the initial condition~\eqref{eq:open initial condition} and the closed free energy~$F^c$.

Let $\tau^o:=\exp(F^c+F^o)$ be the {\it open partition function}. In~\cite{PST14} the authors introduced the following modified operators:
%\unSasha{Introduce differential operators $\mcL_n$, $n\ge -1$, by}
\begin{gather}\label{eq:open Virasoro operator}
\mcL_n := L_n + u^n s \frac{\partial^{n+1}}{\partial s^{n+1}}+\frac{3n+3}{4}u^n\frac{\partial^{n}}{\partial s^{n}},\quad n\ge -1,
\end{gather}
where the operators~$L_n$ were defined in Section~\ref{subsubsection:closed Virasoro}.

\begin{conj}[Open Virasoro conjecture]\label{conjecture:open Virasoro}
The operators $\mathcal{L}_n$, $n\ge -1$, annihilate the open partition function:
\begin{gather}\label{eq:open virasoro}
\mcL_n\tau^o = 0,\quad n\ge -1.
\end{gather}
\end{conj}
\noindent In~\cite{PST14} equations~\eqref{eq:open virasoro} were called the open Virasoro equations. Again it is easy to see that the open free energy~$F^o$ is fully determined by the open
Virasoro equations~\eqref{eq:openkdv}, the initial condition~
$$
\left.F^o\right|_{t_{\ge 0}=0}=u^{-1}\frac{s^3}{6}
$$
and the closed free energy~$F^c$.

From the closed string equation~\eqref{eq:string} it immediately follows that the open string equation~\eqref{gvvt} is equivalent to~\eqref{eq:open virasoro}, for $n=-1$. Moreover, from the equation~$L_0\tau^c=0$ it follows that the open dilaton equation~$\eqref{eq:open dilaton}$ is equivalent to~\eqref{eq:open virasoro}, for $n=0$.

\begin{rmk}
More precisely, in~\cite{PST14} it was conjectured that there exists a definition of open intersection numbers for $g>0$, for which the open KdV and open Virasoro equations hold. The definition was later
given in~\cite{JSRTa}.
\end{rmk}

Although it was not clear at all that the open KdV and the open Virasoro equations are compatible, in~\cite{Bur14a} it was proved that they indeed have a common solution.

For $g=0$ the conjectures were proved in~\cite{PST14}. In~\cite{JSRTa} the conjectures are proved for $g=1$ and the open string~\eqref{gvvt} and the open dilaton~\eqref{eq:open dilaton} equations are proved for all $g.$

The main result of this paper is the following theorem.
\begin{thm}\label{theorem:main}
Conjectures~\ref{conjecture:open KdV} and~\ref{conjecture:open Virasoro} are true.
\end{thm}

\subsubsection{Burgers-KdV hierarchy}

Let $F$ be a power series in the variables $s,t_0,t_1,\ldots$ with the coefficients from $\mbC[u,u^{-1}]$. In~\cite{Bur14a} the following system of equations was introduced:
\begin{align}
\frac{2n+1}{2u^2}F_{t_n}&=\left(\frac{1}{2}\frac{\partial^2}{\partial t_0^2}+F_{t_0}\frac{\partial}{\partial
t_0}+\frac{1}{2}F_{t_0}^2+\frac{1}{2}F_{t_0,t_0}+F^c_{t_0,t_0}\right)F_{t_{n-1}}+\frac{1}{2}F_{t_0}F^c_{t_0,t_{n-1}} + \frac{3}{4}F^c_{t_0,t_0,t_{n-1}},\label{eq:t-flows}\\
F_s& = u\left(\frac{1}{2}F_{t_0}^2+\frac{1}{2}F_{t_0,t_0}+F^c_{t_0,t_0}\right).\label{eq:s-flow}
\end{align}
It was called the half of the Burgers-KdV hierarchy. This system is obviously stronger than the system of the open KdV equations~\eqref{eq:openkdv}. In~\cite{Bur14a} it was actually
shown that the half of the Burgers-KdV hierarchy is equivalent to the open KdV equations together with equation~\eqref{eq:s-flow}.

Denote by~$\tF^o$ a unique solution of system~\eqref{eq:t-flows}-\eqref{eq:s-flow} specified by the initial condition
$$
\tF^o|_{t_{\geq 1}=0,s=0}=0.
$$
In~\cite{Bur14a} it was shown that~$\tF^o$ satisfies the open KdV equations, the initial condition~\eqref{eq:open initial condition} and the open Virasoro equations. This proved the equivalence of the open analog of Witten's conjecture and the open Virasoro conjecture. This also shows that Theorem~\ref{theorem:main} immediately implies the following corollary.
\begin{cor}
The open free energy~$F^o$ satisfies the half of the Burgers-KdV hierarchy.
\end{cor}

Consider more variables $s_1,s_2,\ldots$ and let $s_0:=s$. Let $F$ be a power series in the variables $s_0,s_1,\ldots,,t_0,t_1,\ldots$ with the coefficients from $\mbC[u,u^{-1}]$. Let
us extend the half of the Burgers-KdV hierarchy by the following equations:
\begin{gather}\label{eq:higher s-flows}
\frac{n+1}{u^2}F_{s_{n}}=\left(\frac{1}{2}\frac{\partial^2}{\partial t_0^2}+F_{t_0}\frac{\partial}{\partial
t_0}+\frac{1}{2}F_{t_0}^2+\frac{1}{2}F_{t_0,t_0}+F^c_{t_0,t_0}\right)F_{s_{n-1}},\quad n\ge 1.
\end{gather}
In~\cite{Bur14a} the extended system~\eqref{eq:t-flows}-\eqref{eq:higher s-flows} was called the (full) Burgers-KdV hierarchy. Let~$\tF^{o,ext}$ be a unique solution of it specified by the
initial condition
\[
\tF^{o,ext}|_{t_{\geq 1}=0,s_{\geq 0}=0}=0.
\]
We obviously have~$\tF^{o,ext}|_{s_{\geq 1}=0}=\tF^o$. In~\cite{Bur14b} it was proved that the function $$\ttau^{o,ext}:=\exp(\tF^{o,ext}+F^c)$$ satisfies the following extended Virasoro equations:
\begin{gather}\label{eq:extended Virasoro}
\mcL^{ext}_n\ttau^{o,ext}=0,\quad n\ge -1,
\end{gather}
where
\begin{gather*}
\mcL_n^{ext} := L_n+\sum_{i\geq 0} \frac{(i+n+1)!}{i!}s_i\frac{\partial}{\partial s_{n+i}}+u\frac{3(n+1)!}{4}\frac{\partial}{\partial s_{n-1}}+
\delta_{n,-1}u^{-1}s+\delta_{n,0}\frac{3}{4}.
\end{gather*}
Here we, by definition, put $\frac{\d}{\d s_{-2}}:=\frac{\d}{\d s_{-1}}:=0$.

In~\cite{Bur14a} it was conjectured that, by adding descendents for boundary marked points, one can geometrically define intersection numbers which will be the coefficients of
$\tF^{o,ext}$. In~\cite{JSRTb} J.~Solomon an R.T. give a complete proposal for the construction in all genera, and it is proved to produce the correct intersection numbers in genus $0.$ Moreover, with this proposal the extended open free energy $F^{o,ext}$ is defined, as well as the extended open partition function
\[
\tau^{o,ext}:=\exp(F^c+F^{o,ext}).
\]
It is proved in \cite{JSRTb} that the following equations hold
\[
\frac{\partial}{\partial s_n}\tau^{o,ext} = \frac{u^n}{(n+1)!}\frac{\partial^{n+1}}{\partial s_0^{n+1}}\tau^{o,ext},\quad\text{$n\ge 1$}.
\]
In \cite{Bur14b}, Section 5.2, it is shown that $\ttau^{o,ext}$ satisfies these equations as well. Thus, Theorem~\ref{theorem:main} implies that $\tau^{o,ext}=\ttau^{o,ext}$ and we immediately obtain the following generalization of Theorem~\ref{theorem:main}.
\begin{thm}\label{theorem:extended main}
The extended open free energy $F^{o,ext}$ is a solution of the full Burgers-KdV hierarchy.
\end{thm}
\begin{rmk}
From the recent result of A.~Alexandrov~\cite{Ale14} it also follows that the extended open partition function~$\tau^{o,ext}$ becomes a tau-function of the KP hierarchy after the change of variables $t_n=(2n+1)!!T_{2n+1}$ and $s_n=2^{n+1}(n+1)!T_{2n+2}$.
\end{rmk}

\subsection{Combinatorial formula for the open intersection numbers}\label{ocomb}

In~\cite{RT} R.T. proved a combinatorial formula for the geometric models which were defined in~\cite{PST14,JSRTa}. He showed that all these intersection numbers can be calculated
as sums of amplitudes of diagrams which will be described below. In this paper a matrix model is constructed out of this combinatorial formula. Using this matrix model we prove our main Theorem~\ref{theorem:main}. R.T. also derived an extended formula for the intersection numbers of \cite{JSRTb}, and it will appear in a future paper.

A topological $(g,k,l)$-surface with boundary $\Sigma,$ is a topological connected oriented surface with non-empty boundary, genus $g,~k$ boundary marked points $\{x_i\}_{i\in[k]},$ and $l$ internal marked points $\{z_i\}_{i\in[l]}.$
By genus we mean, as usual in the open theory, the doubled genus, that is, the genus of the doubled surface obtained by gluing two copies of $\Sigma$ along $\partial\Sigma.$
We require the stability condition
\[
2g-2+k+2l>0.
\]

\begin{definition}
Let $g,k,l$ be non-negative integers such that $2g-2+k+2l>0,~A$ be a finite set and $\alpha:[l]\to A$ a map. $\alpha,A$ will be implicit in the definition. A \emph{$(g,k,l)$-ribbon graph with boundary} is an
%\newRan{isomorphism class of embeddings}\unRan{embedding}\footnote{Sasha - is t better like this or before?}
embedding $\iota:G\to\Sigma$ of a connected graph~$G$ into a $(g,k,l)$-surface with boundary~$\Sigma$ such that

\begin{itemize}

\item $\{x_i\}_{i\in [k]}\subseteq \iota(V(G))$, where $V(G)$ is the set of vertices of~$G$. We henceforth consider $\{x_i\}$ as vertices.

\item The degree of any vertex $v\in V(G)\setminus\{x_i\}$ is at least $3$.

\item $\partial\Sigma\subseteq \iota(G)$.

\item If $l \geq 1,$ then $$\Sigma\setminus\iota(G)=\coprod_{i\in [l]} D_i,$$ where each $D_i$ is a topological open disk, with $z_i\in D_i$. We call the disks $D_i$ faces.

\item If $l=0$, then $\iota(G)=\partial\Sigma$.

\end{itemize}
The genus $g(G)$ of the graph~$G$ is the genus of $\Sigma$. The number of the boundary components of~$G$ or~$\Sigma$ is denoted by $b(G)$ and $v_I(G)$ stands for the number of the internal vertices. We denote by~$\Faces(G)$ the set of faces of the graph~$G,$ and we consider $\alpha$ as a map $$\alpha\colon\Faces(G)\to A,$$
by defining for $f\in\Faces(G),~\alpha(f):=\alpha(i),$ where $z_i$ is the unique internal marked point in~$f.$ The map $\alpha$ is called the labeling of $G.$
Denote by~$V_{BM}(G)$ the set of boundary marked points~$\{x_i\}_{i\in [k]}.$

Two ribbon graphs with boundary $\iota\colon G\to\Sigma,~\iota'\colon G'\to\Sigma'$ are isomorphic, if there is an orientation preserving homeomorphism
$\Phi\colon(\Sigma,\{z_i\},\{x_i\})\to(\Sigma',\{z'_i\},\{x'_i\}),$ and an isomorphism of graphs $\phi\colon G\to G'$, such that
\begin{enumerate}
\item
$\iota'\circ\phi = \Phi\circ\iota.$
\item
$\phi(x_i)=x'_i,$ for all $i\in[k].$
\item
$\alpha'(\phi(f))=\alpha(f),$
where $\alpha,\alpha'$ are the labelings of $G,G'$ respectively and~$f\in\Faces(G)$ is any face of the graph~$G.$
\end{enumerate}
Note that in this definition we do not require the map~$\Phi$ to preserve the numbering of the internal marked points.

A ribbon graph is \emph{critical}, if
\begin{itemize}

\item Boundary marked points have degree $2$.

\item All other vertices have degree $3$.

\item If $l=0,$ then $g=0$ and $k=3.$

\end{itemize}
A $(0,3,0)-$ribbon graph with boundary is called a \emph{ghost}.
\end{definition}
In Figure \ref{fig:nonnodal}~two critical ribbon graphs are shown, the right one is a ghost.
We draw internal edges as thick (ribbon) lines, while boundary edges are usual lines.
Note that not all boundary vertices are boundary marked points.
We draw parallel lines inside the ghost, to emphasize that the face bounded by the boundary is a special face, without a marked point inside.
\begin{figure}[t]
\centering
\includegraphics[scale=.8]{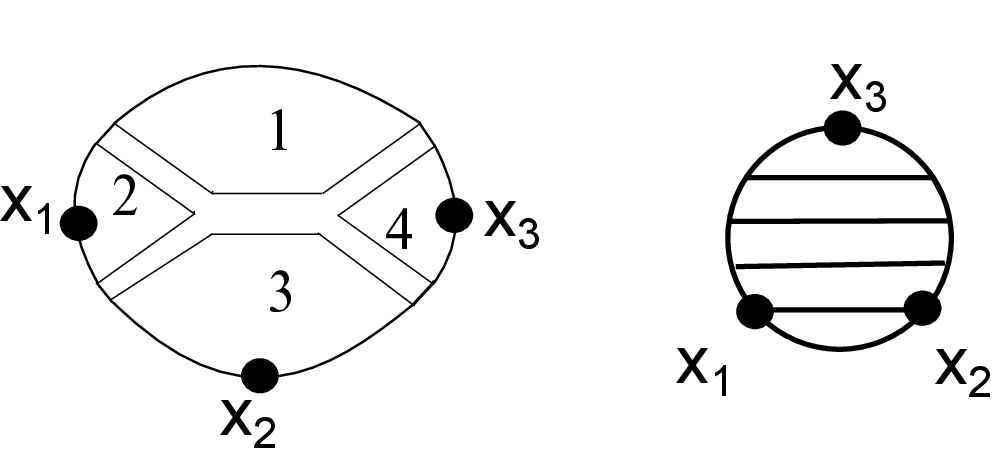}
\caption{Critical ribbon graphs.}
\label{fig:nonnodal}
\end{figure}

\begin{definition}
A \emph{nodal ribbon graph with boundary} is $G=\left(\coprod_i G_i\right)/N$, where
\begin{itemize}
\item $\iota_i\colon G_i\to\Sigma_i$ are ribbon graphs with boundary.
\item $N\subset (\cup_i V(G_i))\times(\cup_i V(G_i))$ is a set of \emph{ordered} pairs of boundary marked points $(v_1,v_2)$ of the $G_i$'s which we identify.
\end{itemize}
We require that
\begin{itemize}
\item $G$ is a connected graph,
\item Elements of $N$ are disjoint as sets (without ordering).
\end{itemize}

After the identification of the vertices~$v_1$ and~$v_2$ the corresponding point in the graph is called a node. The vertex~$v_1$ is called the legal side of the node and the vertex~$v_2$ is called the illegal side of the node.

The set of edges $\Edges(G)$ is composed of the internal edges of the $G_i$'s and of the boundary edges.
The boundary edges are the boundary segments between successive vertices which are not the illegal sides of nodes. For any boundary edge $e$ we denote by $m(e)$ the number of the
illegal sides of nodes lying on it. The boundary marked points of~$G$ are the boundary marked points of~$G_i$'s, which are not nodes. The set of boundary marked points of~$G$ will
be denoted by~$V_{BM}(G)$ also in the nodal case.

A nodal graph~$G=\left(\coprod_i G_i\right)/N$ is \emph{critical}, if
\begin{itemize}

\item All of its components~$G_i$ are critical.

\item Any boundary component of $G_i$ has an odd number of points that are the boundary marked points or the legal sides of nodes.

\item Ghost components do not contain the illegal sides of nodes.

\end{itemize}
A nodal ribbon graph with boundary is naturally embedded into the nodal surface $\Sigma=\left(\coprod_i\Sigma_i\right)/N$. The (doubled) genus of $\Sigma$ is called the genus of the graph. The
notion of an isomorphism is also as in the non-nodal case.
\end{definition}
\begin{rmk}
The genus of a closed, and in particular doubled, nodal surface $\Sigma$ is the genus of the smooth surface obtained by smoothing all nodes of $\Sigma.$
\end{rmk}

In Figure~\ref{nodal}~there is a critical nodal graph of genus~$0$, with~$5$ boundary marked points, $6$ internal marked points, three components, one of them is a ghost, two nodes, where a
plus sign is drawn next to the legal side of a node and a minus sign is drawn next to the illegal side.

\begin{figure}[t]
\centering
\includegraphics[scale=.8]{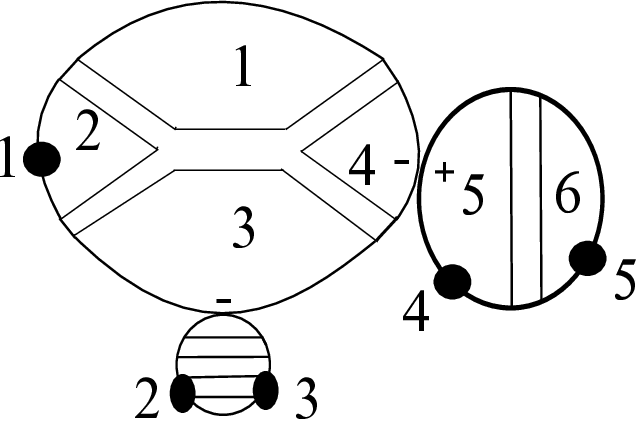}
\caption{A critical nodal ribbon graph.}
\label{nodal}
\end{figure}

In Figure \ref{noncriticalnodal}~a non-critical nodal graph is shown. Here there is some vertex of degree $4,$ the components do not satisfy the parity condition and the ghost component has an illegal node.
\begin{figure}[t]
\centering
\includegraphics[scale=.8]{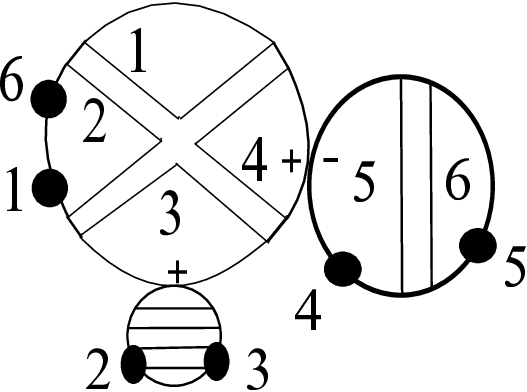}
\caption{A non-critical nodal ribbon graph.}
\label{noncriticalnodal}
\end{figure}

Denote by $G_{g,k,l}$ the set of isomorphism classes of critical nodal ribbon graphs with boundary of genus~$g$, with~$k$ boundary marked points,~$l$ faces and together with a bijective labeling $\alpha:\Faces(G)\stackrel{\sim}{\to} [l],$ %\newRan{induced from the idenity map $\alpha:[l]\to[l].$}
The combinatorial formula in~\cite{RT} is

\begin{thm}
Fix $g,k,l\ge 0$ such that $2g-2+k+2l>0$. Let $\lambda_1,\ldots,\lambda_l$ be formal variables. Then we have
\begin{multline}\label{eq:combinatorial formula}
2^{\frac{g+k-1}{2}}\sum_{a_1,\ldots,a_l\ge 0}\blangle\tau_{a_1} \tau_{a_2} \cdots \tau_{a_l}\sigma^k\brangle_g^o\prod_{i=1}^l\frac{2^{a_i}(2a_i-1)!!}{\lambda_i^{2a_i+1}}=\\
=\sum_{G=\left(\coprod_i G_i\right)/N\in G_{g,k,l}}\frac{\prod_i 2^{v_I(G_i)+g(G_i)+b(G_i)-1}}{|\Aut(G)|}\prod_{e\in\Edges(G)}\lambda(e),
\end{multline}
where
\[
\lambda(e):=
\begin{cases}
\frac{1}{\lambda_i+\lambda_j},&\text{if $e$ is an internal edge between faces $i$ and $j$};\\
\frac{1}{(m+1)}\binom{2m}{m}\lambda_i^{-2m-1},&\text{if $e$ is a boundary edge of face $i$ and $m(e)=m$};\\
1,&\text{if $e$ is a boundary edge of a ghost}.
\end{cases}
\]
\end{thm}

\subsection{Acknowledgments}

We thank A.~Alexandrov, L.~Chekhov, D.~Kazhdan, O.~G.~Louidor, A.~Okounkov, R.~Pandharipande, P.~Rossi, S.~Shadrin, J.~P.~Solomon and D.~Zvonkine for discussions related to the work presented here.

A. B. was supported by grant ERC-2012-AdG-320368-MCSK in the
group of R.~Pandharipande at ETH Zurich and grant RFFI-16-01-00409.
R.T. was supported by ISF Grant 1747/13 and ERC Starting Grant 337560 in the group of J.~P.~Solomon at the Hebrew university of Jerusalem.

Part of the work was completed during the visit of A.B to the Einstein Institute of Mathematics
of the Hebrew University of Jerusalem in 2014 and during the visits of R.T. to the Forschungsinstitut f\"ur Mathematik at ETH Z\"urich in 2013 and 2014.

%%%%%%%%%%%%%%%%%%%%%%%%%%%%%%%%%%%%%%%%%%%%%%%%%%%%%%%%%%%%%%%%%%%%%%%%%%
%%%%%%%%%%%%%%%%%%%%%%%%%%%%%%%%%%%%%%%%%%%%%%%%%%%%%%%%%%%%%%%%%%%%%%%%%%

\section{Matrix model}

In this section we present a matrix integral that is a starting point in our proof of Theorem~\ref{theorem:main}. Instead of deriving a matrix integral for the open partition function~$\tau^o$ directly from the combinatorial formula~\eqref{eq:combinatorial formula}, we write a matrix model for an
auxiliary function~$f^o_N$ that is a sum over non-nodal ribbon graphs with boundary with some additional structure. We then relate the function~$f^o_N$ to the open partition function~$\tau^o$ by an
action of the exponent of some quadratic differential operator.

The section is organized as follows. In Section~\ref{subsection:reformulation} we give a slight reformulation of the combinatorial formula~\eqref{eq:combinatorial formula}. In
Section~\ref{subsection:sum over non-nodal graphs} we introduce an auxiliary function~$f^o_N$ and relate it to the open partition function~$\tau^o$. This relation is given by
Lemma~\ref{lemma:relation}. Section~\ref{subsection:brief recall} contains a brief review of basic facts about the integration over the space of Hermitian matrices. In
Section~\ref{subsection:matrix integral} we give a matrix integral for the function~$f^o_N$. This is the subject of Proposition~\ref{proposition:matrix integral}. The matrix integral in this proposition is understood in the sense of formal matrix integration. In Section~\ref{subsection:convergent matrix integral} we discuss how to make sense of it as a convergent integral.

We fix an integer $N\ge 1$ throughout this section and we set the genus parameter~$u$ to be equal to~$1.$

\subsection{Reformulation of the combinatorial formula}\label{subsection:reformulation}

Here we reformulate the combinatorial formula~\eqref{eq:combinatorial formula}. This step is completely analogous to what M.~Kontsevich
did in~\cite{Kon92} (see the proof of Theorem~1.1 there).

Denote by $G_N^{cr}$ the set of isomorphism classes of critical nodal ribbon graphs with boundary together with a labeling $\alpha\colon\Faces(G)\to [N]$ (a coloring of faces in $N$ colors). For a graph $G\in G_N^{cr}$ and an edge $e\in\Edges(G)$, let
$$
\tlambda(e):=
\begin{cases}
-\frac{1}{\lambda_{\alpha(f_1)}+\lambda_{\alpha(f_2)}},&\text{if $e$ is an internal edge between faces $f_1$ and $f_2$};\\
-\frac{1}{(m+1)}\binom{2m}{m}\lambda_{\alpha(f)}^{-2m-1},&\text{if $e$ is a boundary edge of face $f$ and $m(e)=m$};\\
1,&\text{if $e$ a boundary edge of a ghost}.
\end{cases}
$$
Introduce formal variables $\lambda_1,\ldots,\lambda_N$ and consider the diagonal~$N\times N$ matrix
$$
\Lambda:=\diag(\lambda_1,\ldots,\lambda_N).
$$
From the combinatorial formula~\eqref{eq:combinatorial formula} it follows that
\begin{gather}\label{eq:reformulation}
\left.F^o\right|_{t_i=-(2i-1)!!\tr\Lambda^{-2i-1}}=\sum_{G=\left(\coprod_i G_i\right)/N\in
G^{cr}_N}\frac{c(G)}{|\Aut(G)|}\left(\prod_{e\in\Edges(G)}\tlambda(e)\right)\frac{s^{|V_{BM}(G)|}}{|V_{BM}(G)|!},
\end{gather}
where $c(G):=\prod_i c(G_i)$ and
$$
c(G_i):=
\begin{cases}
\half,&\text{if $G_i$ is a ghost};\\
2^{-e_I(G_i)+2v_I(G_i)-|V_{BM}(G_i)|+b(G_i)},&\text{otherwise}.
\end{cases}
$$
Here $e_I(G_i)$ stands for the number of the internal edges of the graph~$G_i$.

\subsection{Sum over non-nodal graphs}\label{subsection:sum over non-nodal graphs}

Here we introduce an auxiliary function~$f^o_N$ and relate it to the open partition function~$\tau^o$.

We denote by $\mcR_N$ the ring of formal series of the form
$$
f(\lambda_1,\ldots,\lambda_N)=\sum_{m\ge 0}f_m(\lambda_1,\ldots,\lambda_N),
$$
where $f_m$ is a rational function in $\lambda_1,\ldots,\lambda_N$ homogeneous of degree $-m$. We denote by $\mcR_N^{-d}\subset\mcR_N$ the subspace that consists of series of the form $f=\sum_{m\ge d}f_m$, where $f_m$ is a rational function in~$\lambda_1,\ldots,\lambda_N$ homogeneous of degree $-m$. Note that the ring $\mbC[[\lambda_1^{-1},\ldots,\lambda_N^{-1}]]$ can
be naturally considered as a subring of $\mcR_N$.

Let us introduce the following auxiliary set of graphs. Denote by $G_N^{cr,nn}$ the set of isomorphism classes of critical non-nodal ribbon graphs with boundary, that are not ghosts, together with
\begin{itemize}
\item a labeling $\alpha\colon\Faces(G)\to[N]$,
\item a map $\beta\colon V_{BM}(G)\to\{\pm 1\}$;
\end{itemize}
such that on each boundary component of~$G$ the number of the boundary marked points~$v\in V_{BM}(G)$ with $\beta(v)=1$ is odd. Vertices $v\in V_{BM}(G)$ with $\beta(v)=1$ will be called legal and vertices $v\in V_{BM}(G)$ with $\beta(v)=-1$ will be called illegal. The boundary edges of~$G$ are, by definition, the boundary segments between successive vertices of~$G$ which are not the illegal boundary marked points. We will use the following notations:
\begin{align*}
&v_{BM+}(G):=|\{v\in V_{BM}(G)|\beta(v)=1\}|,\\
&v_{BM-}(G):=|\{v\in V_{BM}(G)|\beta(v)=-1\}|.
\end{align*}

Let
$$
\tau^c_N(\Lambda):=\left.\tau^c\right|_{t_i=-(2i-1)!!\tr\Lambda^{-2i-1}}.
$$
Introduce an auxiliary formal variable $s_-$ and define the following series of rational functions:
\begin{gather}\label{eq:definition of f}
f^o_N(\Lambda,s,s_-):=\tau^c_N\exp\left(\sum_{G\in G^{cr,nn}_N}\frac{\tc(G)}{|\Aut(G)|}\left(\prod_{e\in\Edges(G)}\tlambda(e)\right)\frac{s^{v_{BM+}(G)}}{v_{BM+}(G)!}\frac{s_-^{v_{BM-}(G)}}{v_{BM-}(G)!}\right),
\end{gather}
where
\begin{gather}\label{eq:new coefficient}
\tc(G):=2^{-e_I(G)+2v_I(G)-v_{BM+}(G)+b(G)}.
\end{gather}
By definition, $f^o_N$ is an element of $\mcR_N[[s,s_-]]$.

\begin{rmk}
We do not know  if $f^o_N$ belongs to $\mbC[[\lambda_1^{-1},\ldots,\lambda_N^{-1}]][[s,s_-]]$ or not. Several computations in low degrees motivate us to conjecture that
$f^o_N\in\mbC[[\lambda_1^{-1},\ldots,\lambda_N^{-1}]][[s,s_-]]$. Moreover, we conjecture that there exists a power series $f^o(t_0,t_1,\ldots,s,s_-)\in\mbC[[t_0,t_1,\ldots]][[s,s_-]]$
such that for any $N\ge 1$ we have
$$
\left.f^o\right|_{t_i=-(2i-1)!!\tr\Lambda^{-2i-1}}=f^o_N.
$$
However, we do not need this statement in the paper.
\end{rmk}

Let
$$
\tau^o_N(\Lambda,s):=\left.\tau^o\right|_{t_i=-(2i-1)!!\tr\Lambda^{-2i-1}}.
$$

\begin{lemma}\label{lemma:relation}
We have
\begin{gather}\label{eq:relation}
\tau^o_N=\left.e^{\half\frac{\d^2}{\d s\d s_-}}\left(e^{\frac{s^3}{6}}f^o_N\right)\right|_{s_-=0}.
\end{gather}
\end{lemma}
\begin{proof}
The lemma easily follows from equation~\eqref{eq:reformulation}, the definition~\eqref{eq:definition of f} of~$f^o_N$ and the standard result about the graphical representation of the action of the exponent of a quadratic differential operator. Note that the factor $e^{\frac{s^3}{6}}$ corresponds to ghost components in critical nodal ribbon graphs.
\end{proof}

\subsection{Brief recall of the matrix integration}\label{subsection:brief recall}

We recommend the book~\cite[Sections~3,~4]{LZ04} as a good introduction to this subject.

\subsubsection{Gaussian measure on the space of Hermitian matrices}

Denote by $\mcH_N$ the space of Hermitian~$N\times N$ matrices. For $H=(h_{i,j})\in\mcH_N$ let $x_{i,i}:=h_{i,i}$ and $x_{i,j}:=\Re(h_{i,j}), y_{i,j}:=\Im(h_{i,j}), i<j$. Introduce a
volume form~$dH$ by
$$
d H:=\prod_{i=1}^N d x_{i,i}\prod_{1\le i<j\le N}d x_{i,j}d y_{i,j}.
$$
Consider positive real numbers $\lambda_1,\lambda_2,\ldots,\lambda_N\in\mbR_{>0}$. Recall that $\Lambda:=\diag(\lambda_1,\ldots,\lambda_N)$. Introduce a Gaussian measure
$d\mu_{\Lambda,N}(H)$ on the space~$\mcH_N$ by
\begin{gather*}
d\mu_{\Lambda,N}(H):=c_{\Lambda,N}e^{-\half\tr H^2\Lambda}dH,
\end{gather*}
where the normalization
$$
c_{\Lambda,N}:=(2\pi)^{-N^2/2}\prod_{i=1}^N\lambda_i^{1/2}\prod_{1\le i<j\le N}(\lambda_i+\lambda_j)
$$
is determined by the constraint
$$
\int_{\mcH_N}d\mu_{\Lambda,N}(H)=1.
$$

\subsubsection{Wick formula}

For any polynomial $f(H)\in\mbC[h_{i,j}]_{1\le i,j\le N}$ let
$$
\<f(H)\>_{\Lambda,N}:=\int_{\mcH_N}f(H)d\mu_{\Lambda,N}(H).
$$
The integrals $\<f\>_{\Lambda,N}$ are described by the following result (see e.g.~\cite[Section 3.2.3]{LZ04}).

\begin{lemma}[Wick formula]\label{lemma:Wick formula}
{\it 1}. If $f(H)$ is homogeneous of odd degree, then $\<f\>_{\Lambda,N}=0$.\\
{\it 2}. For any indices $1\le i_1,\ldots,i_{2k},j_1,\ldots,j_{2k}\le N$, we have
$$
\<h_{i_1,j_1}h_{i_2,j_2}\ldots
h_{i_{2k},j_{2k}}\>_{\Lambda,N}=\sum\<h_{i_{p_1},j_{p_1}}h_{i_{q_1},j_{q_1}}\>_{\Lambda,N}\<h_{i_{p_2},j_{p_2}}h_{i_{q_2},j_{q_2}}\>_{\Lambda,N}\ldots\<h_{i_{p_k},j_{p_k}}h_{i_{q_k},j_{q_k}}\>_{\Lambda,N},
$$
where the sum is taken over all permutations $p_1q_1p_2q_2\ldots p_kq_k$ of the set of indices~$1,2,\ldots,2k$ such that $p_1<p_2<\ldots<p_k$ and $p_i<q_i$.\\
{\it 3}. We have
\begin{gather*}
\<h_{i,j}h_{k,l}\>_{\Lambda,N}=
\begin{cases}
\frac{2}{\lambda_i+\lambda_j},&\text{if $j=k$ and $i=l$};\\
0,&\text{otherwise}.
\end{cases}
\end{gather*}
\end{lemma}

%%%%%%%%%%%%%%%%%%%%%%%%%%%%%%%%%%%%%%%%%%%%%%%%%%%%%%%%%

\subsection{Matrix integral for $f^o_N$}\label{subsection:matrix integral}

In this section we present a matrix integral representation for the function $f^o_N$.

Let
\begin{align*}
&G(\Lambda,s_-):=\sum_{m\ge 0}\frac{1}{m+1}{2m\choose m}s_-^m\Lambda^{-2m-1},\\
&I_B(H,\Lambda,s,s_-):=\tr\left[\sum_{m\ge 1}\frac{1}{m}\left(\left(\frac{iH-s}{2}G(\Lambda,s_-)\right)^m-\left(\frac{iH+s}{2}G(\Lambda,s_-)\right)^m\right)\right].
\end{align*}
Note that
$$
G(\Lambda,s_-)=\frac{2}{\Lambda+\sqrt{\Lambda^2-4s_-}}.
$$
Therefore, the exponent~$e^{I_B}$ can be represented in the following way
$$
e^{I_B}=\det\left(\frac{\Lambda+\sqrt{\Lambda^2-4s_-}-iH-s}{\Lambda+\sqrt{\Lambda^2-4s_-}-iH+s}\right).
$$
\begin{prop}\label{proposition:matrix integral}
We have
\begin{gather}\label{eq:matrix integral for f^o}
f^o_N(\Lambda,s,s_-)=c_{\Lambda,N}\int^{form}_{\mcH_N} e^{\frac{i}{6}\tr H^3+I_B-\frac{1}{2}\tr H^2\Lambda} dH.
\end{gather}
\end{prop}
\noindent Here the integral is understood in the sense of formal matrix integration. It means the following. We express $e^{\frac{i}{6}\tr H^3+I_B}$ as a series of the form
$$
e^{\frac{i}{6}\tr H^3+I_B}=\sum_{a,b,m\ge 0}s^a s_-^bP_{a,b,m},
$$
where $P_{a,b,m}$ is a polynomial of degree $m$ in $\tr H^3$ and expressions of the form
$$
\tr(H\Lambda^{-d_1}H\Lambda^{-d_2}\ldots H\Lambda^{-d_r}),\quad\text{where $d_1,\ldots,d_r\ge 1$}.
$$
Here the degree is introduced by putting
$$
\deg(\tr H^3):=3,\qquad\deg(\tr(H\Lambda^{-d_1}H\Lambda^{-d_2}\ldots H\Lambda^{-d_r})):=r+2\sum_{i=1}^r d_i.
$$
From Lemma~\ref{lemma:Wick formula} it follows that
$$
c_{\Lambda,N}\int_{\mcH_N}P_{a,b,m}e^{-\frac{1}{2}\tr H^2\Lambda}dH=\<P_{a,b,m}\>_{\Lambda,N}
$$
is zero, if $m$ is odd, and is a rational function in $\lambda_1,\ldots,\lambda_N$ of degree $-\frac{m}{2}$, if $m$ is even. Then the integral on the right-hand side
of~\eqref{eq:matrix integral for f^o} is defined by
$$
c_{\Lambda,N}\int^{form}_{\mcH_N} e^{\frac{i}{6}\tr H^3-\frac{1}{2}\tr H^2\Lambda+I_B}dH:=\sum_{a,b,m\ge 0}s^a s_-^b\<P_{a,b,2m}\>_{\Lambda,N}.
$$
\begin{proof}[Proof of Proposition~\ref{proposition:matrix integral}]
We begin by recalling the derivation of the matrix integral for $\tau^c_N$ from Kontsevich's combinatorial formula~\eqref{eq:Kontsevich's formula}~(\cite{Kon92}). Consider an arbitrary trivalent ribbon graph without boundary with a coloring of faces in~$N$ colors. It can be obtained by gluing trivalent stars (see Fig.~\ref{fig:internal vertex}) in such a way that corresponding indices on glued edges coincide.
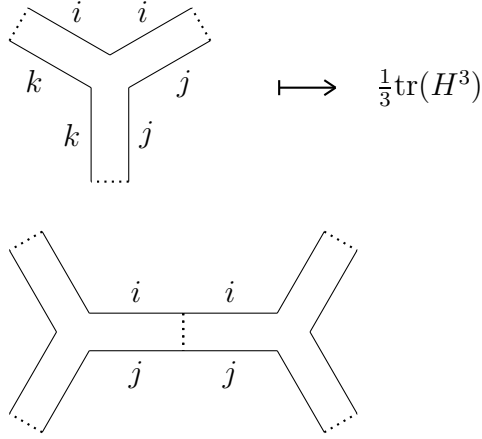
\begin{figure}[t]

\begin{tikzpicture}[scale=0.5]

\begin{scope}[shift={(0.55,6.5)}]
\draw (0.5,0) -- (0.5,-2.5);
\draw (-0.5,0) -- (-0.5,-2.5);
\draw (0.5,0) -- (2.67,1.25);
\draw (0,0.87) -- (2.17,2.12);
\draw (-0.5,0) -- (-2.67,1.25);
\draw (0,0.87) -- (-2.17,2.12);

\draw [thick,dotted] (0.5,-2.5) -- (-0.5,-2.5);
\draw [thick,dotted] (2.67,1.25) -- (2.17,2.12);
\draw [thick,dotted] (-2.67,1.25) -- (-2.17,2.12);

\draw [thick] (4.5,0) -- (6,0);
\draw [thick] (4.5,0.2) -- (4.5,-0.2);
\draw [thick] (6,0) -- (5.7,0.2);
\draw [thick] (6,0) -- (5.7,-0.2);

\coordinate [label=0: $j$] (B) at (0.5,-1.25);
\coordinate [label=180: $k$] (B) at (-0.5,-1.25);
\coordinate [label=-45: $j$] (B) at (1.5,0.75);
\coordinate [label=135: $i$] (B) at (1.30,1.45);
\coordinate [label=-135: $k$] (B) at (-1.5,0.75);
\coordinate [label=45: $i$] (B) at (-1.3,1.45);
\coordinate [label=center: $\frac{1}{3}\mathrm{tr}(H^3)$] (B) at (8.5,0);
\end{scope}

\draw (0,0.5) -- (2.5,0.5);
\draw (0,-0.5) -- (2.5,-0.5);
\draw (0,0.5) -- (-1.25,2.67);
\draw (-0.87,0) -- (-2.12,2.17);
\draw (0,-0.5) -- (-1.25,-2.67);
\draw (-0.87,0) -- (-2.12,-2.17);

\draw [thick,dotted] (2.5,0.5) -- (2.5,-0.5);
\draw [thick,dotted] (-1.25,2.67) -- (-2.12,2.17);
\draw [thick,dotted] (-1.25,-2.67) -- (-2.12,-2.17);

\coordinate [label=90: $i$] (B) at (1.25,0.5);
\coordinate [label=-90: $j$] (B) at (1.25,-0.5);

\begin{scope}[shift={(5,0)}]
\draw (0,0.5) -- (-2.5,0.5);
\draw (0,-0.5) -- (-2.5,-0.5);
\draw (0,0.5) -- (1.25,2.67);
\draw (0.87,0) -- (2.12,2.17);
\draw (0,-0.5) -- (1.25,-2.67);
\draw (0.87,0) -- (2.12,-2.17);

\draw [thick,dotted] (-2.5,0.5) -- (-2.5,-0.5);
\draw [thick,dotted] (1.25,2.67) -- (2.12,2.17);
\draw [thick,dotted] (1.25,-2.67) -- (2.12,-2.17);

\coordinate [label=90: $i$] (B) at (-1.25,0.5);
\coordinate [label=-90: $j$] (B) at (-1.25,-0.5);
\end{scope}

\end{tikzpicture}
\caption{Internal vertex}
\label{fig:internal vertex}
\end{figure}
To a trivalent star we associate the polynomial~$\frac{1}{3}\tr H^3\in\mbC[h_{i,j}]$. Then the Wick formula (Lemma~\ref{lemma:Wick formula}) and Kontsevich's formula~\eqref{eq:Kontsevich's formula} imply that
$$
\tau^c_N=c_{\Lambda,N}\int^{form}_{\mcH_N}e^{\frac{i}{6}\tr H^3-\frac{1}{2}\tr H^2\Lambda}dH.
$$
This is the famous Kontsevich integral (\cite{Kon92}). We recommend the reader the book~\cite[Sections~3 and~4]{LZ04} for a more detailed explanation of this technique.

Consider now a non-nodal critical ribbon graph with boundary~$G\in G^{cr,nn}_N$. Each boundary component is a circle with a configuration of vertices of degrees $2$ and $3$. Vertices of degree~$2$ are boundary marked points and they are of two types: legal and illegal. In Fig.~\ref{pic2} we draw an example of a boundary component. Legal boundary marked points are marked by~$+$ and illegal ones by~$-$. Each graph from~$G_N^{cr,nn}$ can be obtained by gluing trivalent stars and some number of boundary components.

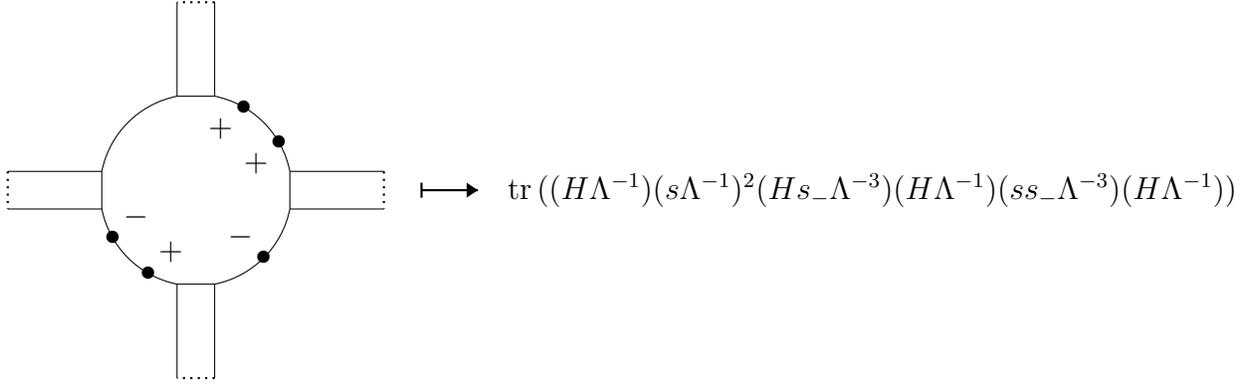
\begin{figure}[t]

\begin{tikzpicture}[scale=0.5]

\draw (2.5,0.5) arc (11.31:78.69:2.55);
\draw (2.5,0.5) -- (5,0.5);
\draw (2.5,-0.5) arc (-11.31:-78.69:2.55);
\draw (2.5,-0.5) -- (5,-0.5);
\draw (2.5,-0.5) -- (2.5,0.5);

\draw (-2.5,0.5) arc (168.69:101.31:2.55);
\draw (-2.5,0.5) -- (-5,0.5);
\draw (-2.5,-0.5) arc (-168.69:-101.31:2.55);
\draw (-2.5,-0.5) -- (-5,-0.5);
\draw (-2.5,-0.5) -- (-2.5,0.5);

\draw (0.5,2.5) -- (0.5,5);
\draw (-0.5,2.5) -- (-0.5,5);
\draw (-0.5,2.5) -- (0.5,2.5);
\draw (0.5,-2.5) -- (0.5,-5);
\draw (-0.5,-2.5) -- (-0.5,-5);
\draw (-0.5,-2.5) -- (0.5,-2.5);

\draw [thick,dotted] (0.5,5) -- (-0.5,5);
\draw [thick,dotted] (-5,0.5) -- (-5,-0.5);
\draw [thick,dotted] (0.5,-5) -- (-0.5,-5);
\draw [thick,dotted] (5,0.5) -- (5,-0.5);

\coordinate [label=center:\textbullet] (B) at (2.21,1.27);
\coordinate [label=-135:$+$] (B) at (2.21,1.27);
\coordinate [label=center:\textbullet] (B) at (1.27,2.21);
\coordinate [label=-135:$+$] (B) at (1.27,2.21);
\coordinate [label=center:\textbullet] (B) at (1.80,-1.80);
\coordinate [label=135:$-$] (B) at (1.80,-1.80);
\coordinate [label=center:\textbullet] (B) at (-1.27,-2.21);
\coordinate [label=45:$+$] (B) at (-1.27,-2.21);
\coordinate [label=center:\textbullet] (B) at (-2.21,-1.27);
\coordinate [label=45:$-$] (B) at (-2.21,-1.27);

\draw [thick] (6,0) -- (7.5,0);
\draw [thick] (6,0.2) -- (6,-0.2);
\fill (7.2,0.2)--(7.5,0)--(7.2,-0.2)--cycle;

\coordinate [label=0:$\mathrm{tr}\left((H\Lambda^{-1})(s\Lambda^{-1})^2(Hs_-\Lambda^{-3})(H\Lambda^{-1})(s s_-\Lambda^{-3})(H\Lambda^{-1})\right)$] (B) at (8,0);

\end{tikzpicture}

\caption{Boundary component}
\label{pic2}

\end{figure}

Consider a boundary component of some non-nodal critical ribbon graph with boundary~$G\in G_N^{cr,nn}$. Recall that, by definition, the boundary edges are the boundary segments between successive vertices which are not illegal boundary marked points. Let $e_1,e_2,\ldots,e_p$ be the edges of the boundary component, ordered in the clockwise direction. Moreover we orient each edge in the clockwise direction, so that any boundary edge points from a source vertex to a target vertex. Let
$$
bp(e_i):=
\begin{cases}
H,&\text{if the source of $e_i$ is of degree $3$},\\
s,&\text{if the source of $e_i$ is of degree $2$}.
\end{cases}
$$
The combinatorial formula~\eqref{eq:definition of f} suggests that to the boundary component we should assign the following expression:
\begin{gather}\label{eq:boundary expression}
\frac{1}{|\Aut|}\tr\left[\prod_{i=1}^p \left(bp(e_i)\frac{1}{m(e_i)+1}{2m(e_i)\choose m(e_i)}s_-^{m(e_i)}\Lambda^{-2m(e_i)-1}\right)\right]\in\mbC[h_{i,j}],
\end{gather}
where $\Aut$ denotes the automorphism group of the boundary component. An example is shown in~Fig.~\ref{pic2}. For a moment we ignore the combinatorial coefficient~\eqref{eq:new coefficient}. If we sum expressions~\eqref{eq:boundary expression} over all possible choices of a boundary component, we get
\begin{gather}\label{eq:tmp expression}
\frac{1}{2}\tr\left[\sum_{k\ge 1}\frac{1}{k}\left((H+s)G(\Lambda,s_-)\right)^k-\sum_{k\ge 1}\frac{1}{k}\left((H-s)G(\Lambda,s_-)\right)^k\right].
\end{gather}
Let us deal more carefully with the combinatorial coefficient~\eqref{eq:new coefficient}. Denote by~$v_{B3}(G)$ the number of the boundary vertices of degree~$3$ of the graph~$G$ and by~$e_B(G)$ the number of the boundary edges. Since $3v_I(G)+v_{B3}(G)=2e_I(G)$, we have
\begin{align*}
\tc(G)=&2^{-e_I(G)+2v_I(G)-v_{BM+}(G)+b(G)}=2^{e_I(G)-v_I(G)-v_{B3}(G)-v_{BM+}(G)+b(G)},\\
(-1)^{|\Edges(G)|}=&(-1)^{e_I(G)}(-1)^{e_B(G)}=i^{3v_I(G)+v_{B3}(G)}(-1)^{v_{B3}(G)+v_{BM+}(G)}=i^{v_I(G)+v_{B3}(G)}(-1)^{v_{BM+}(G)}.
\end{align*}
Therefore, we have to rescale expression~\eqref{eq:tmp expression} in the following way:
\begin{gather*}
\tr\left[\sum_{k\ge 1}\frac{1}{k}\left(\frac{iH-s}{2}G(\Lambda,s_-)\right)^k-\sum_{k\ge 1}\frac{1}{k}\left(\frac{iH+s}{2}G(\Lambda,s_-)\right)^k\right].
\end{gather*}
We immediately recognize here the function~$I_B(H,\Lambda,s,s_-)$. Again, the Wick formula together with Kontsevich's formula~\eqref{eq:Kontsevich's formula} and the combinatorial formula~\eqref{eq:definition of f} imply that
$$
f^o_N=c_{\Lambda,N}\int_{\mcH_N}^{form}e^{\frac{i}{6}\tr H^3+I_B-\frac{1}{2}\tr H^2\Lambda}dH.
$$
The proposition is proved.
\end{proof}

\subsection{Convergent matrix integral}\label{subsection:convergent matrix integral}

One can show that the integral
$$
c_{\Lambda,N}\int_{\mcH_N}e^{\frac{i}{6}\tr H^3-\frac{1}{2}\tr H^2\Lambda}dH
$$
is absolutely convergent and determines a well-defined function of $(\lambda_1,\ldots,\lambda_N)\in\mbR^N_{>0}$. It is not hard to show that the asymptotic expansion of it, when $\lambda_i\to\infty$, is given by~$\tau^c_N$ (\cite{Kon92}).

In our case we can also make sense of the integral in~\eqref{eq:matrix integral for f^o} as a convergent integral. Suppose that~$\lambda_i$ are positive real numbers, $s_-$ is a real number such that $s_-<\frac{\lambda_i^2}{4}$, $i=1,\ldots,N$, and~$s$ is a purely imaginary complex number. Then the integral
$$
c_{\Lambda,N}\int_{\mcH_N} e^{\frac{i}{6}\tr H^3-\frac{1}{2}\tr H^2\Lambda+I_B}dH=c_{\Lambda,N}\int_{\mcH_N} e^{\frac{i}{6}\tr H^3-\frac{1}{2}\tr
H^2\Lambda}\det\left(\frac{\Lambda+\sqrt{\Lambda^2-4s_-}-iH-s}{\Lambda+\sqrt{\Lambda^2-4s_-}-iH+s}\right)dH
$$
is absolutely convergent and determines a well-defined function of $\lambda_1,\ldots,\lambda_N,s$ and $s_-$.

%%%%%%%%%%%%%%%%%%%%%%%%%%%%%%%%%%%%%%%%%%%%%%%%%%%%%%%%%%%%%%%%%%%%%%%%%
%%%%%%%%%%%%%%%%%%%%%%%%%%%%%%%%%%%%%%%%%%%%%%%%%%%%%%%%%%%%%%%%%%%%%%%%%

\section{Formal Fourier transform of $\tau^o$}

In this section we introduce a certain version of the Fourier transform. It happens that after this transformation formula~\eqref{eq:relation} becomes much simpler. The main result of this section is Proposition~\ref{proposition:convergent Fourier transform}.

We again fix an integer $N\ge 1$ throughout this section and set the genus parameter~$u$ to be equal to~$1$.

\subsection{Formal Fourier transform}

Here we define our version of the Fourier transform and describe its main properties. Section~\ref{subsubsection:preliminary} is preliminary. The main definition is contained in Section~\ref{subsubsection:first Fourier transform}. In Section~\ref{subsubsection:second Fourier transform} there is a slightly different version of it that will also be useful. Section~\ref{subsubsection:properties} is devoted to the properties of our Fourier transform.

\subsubsection{Fourier transform for $\mbC[[s]]$}\label{subsubsection:preliminary}

For an arbitrary positive real $z\in\mbR_{>0}$ and a non-negative integer $m\in\mbZ_{\ge 0}$, we have the following classical formula:
$$
\sqrt\frac{z}{2\pi}\int_{\mbR}s^me^{-\half s^2 z}ds=
\begin{cases}
0,&\text{if $m$ is odd};\\
\frac{(m-1)!!}{z^{m/2}},&\text{if $m$ is even}.
\end{cases}
$$

A power series $f\in\mbC[[s]]$ is called admissible if it has the form $f=e^{\frac{s^3}{6}}\tf(s)$, where $\tf(s)\in\mbC[s]$. The space of admissible power series is denoted by $\mbC[[s]]^{adm}\subset\mbC[[s]]$. For an admissible power series $f\in\mbC[[s]]^{adm}$, $f=e^{\frac{s^3}{6}}\tf$, the formal Fourier transform $\Phi^{form}_s[f](z)\in\mbC[[z^{-1},z]$ is defined by
\begin{multline*}
\Phi_s^{form}[f](z):=\sqrt\frac{z}{2\pi}\int^{form}_{\mbR}\tf(-is+z)e^{\frac{i}{6}s^3-\half s^2 z}ds:=\\
=\sum_{k\ge 0}\sqrt\frac{z}{2\pi}\int_{\mbR}\frac{\tf(-is+z)}{k!}\left(\frac{i}{6}s^3\right)^ke^{-\half s^2
z}ds\in\mbC[[z^{-1},z].
\end{multline*}
The reader can see that the formal Fourier transform of~$f$ is the asymptotic expansion of the integral
\begin{gather*}
\sqrt\frac{z}{2\pi}\int_{\mbR}\tf(-is+z)e^{\frac{i}{6}s^3-\half s^2 z}ds,
\end{gather*}
when $z$ goes to $+\infty$.

\subsubsection{Fourier transform for $\mcR_N[[s]]$}\label{subsubsection:first Fourier transform}

An element $f(\Lambda,s)\in\mcR_N[[s]]$ will be called admissible, if the series $\tf=fe^{-\frac{s^3}{6}}=\sum_{i\ge 0}s^i\tf_i$, $\tf_i\in\mcR_N$, satisfies the following property.
There exists a sequence~$d_i$, $i\ge 0$, of positive integers such that $d_i\to\infty$, when $i\to\infty$, and $\tf_i\in\mcR_N^{-d_i}$. The space of all admissible elements will be
denoted by~$\mcR_N[[s]]^{adm}$.

For an admissible element $f\in\mcR_N[[s]]^{adm}$, it can be seen that the element~$\tf=e^{-\frac{s^3}{6}}f$ can be expressed in the following way:
$$
\tf(\Lambda,s)=\sum_m g_m(\Lambda)P_m(s),
$$
where $g_m\in\mbC(\lambda_1,\ldots\lambda_N)$ are homogeneous linearly independent rational functions and $P_m(s)\in\mbC[s]$ are polynomials in~$s$. The formal Fourier
transform~$\Phi^{form}_s[f](\Lambda,z)\in\mcR_N[[z^{-1},z]]$ is defined by
\begin{gather*}
\Phi^{form}_s[f](\Lambda,z):=\sum_m g_m(\Lambda)\sqrt{\frac{z}{2\pi}}\int^{form}_{\mathbb R}e^{\frac{is^3}{6}-\frac{1}{2}s^2z}P_m(-is+z)ds\in\mcR_N[[z^{-1},z]].
\end{gather*}
We again see that the formal Fourier transform can be considered as the asymptotic expansion of the integral
$$
\sqrt\frac{z}{2\pi}\int_{\mbR}\tf(\Lambda,-is+z)e^{\frac{i}{6}s^3-\half s^2 z}ds,
$$
when $z$ goes to $+\infty$.

\subsubsection{Fourier transform for $\mbC[[t_0,t_1,t_2,\ldots]][[s]]$}\label{subsubsection:second Fourier transform}

A formal Fourier transform for power series from $\mbC[[t_0,t_1,\ldots]][[s]]$ is introduced completely analogously. We introduce a grading in the ring~$\mbC[[t_0,t_1,\ldots]]$ assigning to~$t_i$ the degree~$2i+1$. For $d\ge 0$ let $\mbC[[t_0,t_1,\ldots]]^d$ be the subspace of $\mbC[[t_0,t_1,\ldots]]$ that consists of power series of the form $
f=\sum_{m\ge d}f_m$, where $f_m\in\mbC\left[t_0,t_1,\ldots,t_{\lfloor\frac{m-1}{2}\rfloor}\right]$ is a homogeneous polynomial of degree~$m$. The subspace $\mbC[[t_0,t_1,\ldots]][[s]]^{adm}\subset \mbC[[t_0,t_1,\ldots]][[s]]$ is defined similarly to the previous section, using the filtration
$$
\mbC[[t_0,t_1,\ldots]]=\mbC[[t_0,t_1,\ldots]]^0\supset\mbC[[t_0,t_1,\ldots]]^1\supset\mbC[[t_0,t_1,\ldots]]^2\supset\ldots.
$$
It is easy to see that if $f\in\mbC[[t_0,t_1,\ldots]][[s]]$ is admissible, then
$$
f|_{t_i=-(2i-1)!!\tr\Lambda^{-2i-1}}\in\mbC[[\lambda_1^{-1},\ldots,\lambda_N^{-1}]]
$$
is also admissible.

For an admissible element $f\in\mbC[[t_0,t_1,\ldots]][[s]]^{adm}$ the formal Fourier transform is defined by
\begin{gather*}
\Phi^{form}_s[f](t_*,z):=\sqrt{\frac{z}{2\pi}}\int^{form}_{\mathbb R}e^{\frac{is^3}{6}-\frac{1}{2}s^2z}\tf(t_*,-is+z)ds\in\mbC[[t_0,t_1,\ldots]][[z^{-1},z]],
\end{gather*}
where $\tf:=e^{-\frac{s^3}{6}}f$.

The relation to the Fourier transform from the previous section is given by
\begin{gather}\label{eq:relation of two transforms}
\left.\Phi^{form}_s[f]\right|_{t_i=-(2i-1)!!\tr\Lambda^{-2i-1}}=\Phi^{form}_s\left[\left.f\right|_{t_i=-(2i-1)!!\tr\Lambda^{-2i-1}}\right].
\end{gather}

\subsubsection{Properties}\label{subsubsection:properties}

Let us describe three basic properties of our formal Fourier transform. The symbol $\mcA$ will stand for~$\mcR_N$ or~$\mbC[[t_0,t_1,\ldots]]$.

\begin{lemma}\label{lemma:injectivity}
The map $\Phi^{form}_s\colon\mcA[[s]]^{adm}\to\mcA[[z^{-1},z]]$ is injective.
\end{lemma}
\begin{proof}
Suppose that $\mcA=\mcR_N$. The proof for $\mcA=\mbC[[t_0,t_1,\ldots]]$ is the same. Suppose we have $\Phi^{form}_s[f]=0$ for some admissible $f\in\mcR_N[[s]]^{adm}$. Let $\tf=e^{-\frac{s^3}{6}}f$. We have
$$
\tf=\sum_{m}g_m(\Lambda)P_m(s),
$$
where $g_m(\Lambda)\in\mbC(\lambda_1,\ldots,\lambda_N)$ are linearly independent homogeneous rational functions and $P_m(s)\in\mbC[s]$ are polynomials in $s$. Since
$\Phi^{form}_s[f]=0$, we get
$$
\sqrt{\frac{z}{2\pi}}\int^{form}_{\mathbb R}e^{\frac{is^3}{6}-\frac{1}{2}s^2z}P_m(-is+z)ds=0,
$$
for all $m$. Let us choose a non-zero $P_m$ and let $P_m=\sum_{i=0}^k a_is^i$, where $a_k\ne 0$. A direct calculation shows that
$$
\sqrt{\frac{z}{2\pi}}\int^{form}_{\mathbb R}e^{\frac{is^3}{6}-\frac{1}{2}s^2z}P_m(-is+z)ds=a_kz^k+\sum_{i<k}b_iz^i.
$$
Therefore this integral is non-zero. This contradiction proves the lemma.
\end{proof}

\begin{lemma}\label{lemma:Fourier of derivative}
For any admissible element $f\in\mcA[[s]]^{adm}$, the derivative $\frac{\d f}{\d s}$ is also admissible and $\Phi^{form}_s\left[\frac{\d f}{\d
s}\right]=\frac{z^2}{2}\Phi^{form}_s[f]$.
\end{lemma}
\begin{proof}
The proof is again presented in the case $\mcA=\mcR_N$. If $f=e^{\frac{s^3}{6}}\tf$, then $\frac{\d f}{\d s}=e^{\frac{s^3}{6}}\left(\frac{s^2}{2}\tf+\frac{\d\tf}{\d s}\right)$.
Therefore, $\frac{\d f}{\d s}$ is also admissible. Now we compute
\begin{multline*}
\Phi^{form}_s[f_s](\Lambda,z)=\sqrt\frac{z}{2\pi}\int^{form}_{\mbR}e^{\frac{is^3}{6}-\frac{1}{2}s^2z}\left(\frac{(-is+z)^2}{2}\tf(\Lambda,-is+z)+i\frac{\d}{\d
s}\tf(\Lambda,-is+z)\right)ds=\\
=\frac{z^2}{2}\Phi^{form}_s[f](\Lambda,z).
\end{multline*}
The last equality follows from integration by parts. The lemma is proved.
\end{proof}

\begin{lemma}\label{lemma:Fourier and multiplication}
For any admissible element $f\in\mcA[[s]]^{adm}$, the product $s f$ is also admissible and
$$
\Phi^{form}_s[s f]=-\left(\frac{1}{z}\frac{\d}{\d z}-\frac{1}{2z^2}-z\right)\Phi^{form}_s[f].
$$
\end{lemma}
\begin{proof}
The admissibility of the product $sf$ is obvious. We compute
\begin{multline*}
\left(\frac{1}{z}\frac{\d}{\d z}-\frac{1}{2 z^2}\right)\Phi_s^{form}[f]=\sqrt\frac{z}{2\pi}\frac{1}{z}\int^{form}_\mbR\left[\left(i\frac{\d}{\d s}-\frac{s^2}{2}\right)\tf(-is+z)\right]e^{\frac{i}{6}s^3-\frac{1}{2}s^2 z}ds=\\
=\sqrt\frac{z}{2\pi}\int^{form}_\mbR is\tf(-is+z)e^{\frac{i}{6}s^3-\frac{1}{2}s^2 z}ds.
\end{multline*}
Therefore,
$$
\left(\frac{1}{z}\frac{\d}{\d z}-\frac{1}{2z^2}-z\right)\Phi^{form}_s[f]=\sqrt\frac{z}{2\pi}\int^{form}_\mbR(is-z)\tf(-is+z)e^{\frac{i}{6}s^3-\frac{1}{2}s^2 z}ds=\Phi_s^{form}[sf].
$$
The lemma is proved.
\end{proof}

%%%%%%%%%%%%%%%%%%%%%%%%%%%%%%%%%%%%%%%%%%%%%%%%%%%%%%%%%%%%%%%%%%

\subsection{Fourier transform of $\tau^o$}

Recall that
$$
\tau^o_N(\Lambda,s):=\left.\tau^o\right|_{t_i=-(2i-1)!!\tr\Lambda^{-2i-1}}\in\mbC[[\lambda_1^{-1},\ldots,\lambda_N^{-1}]][[s]].
$$
From formula~\eqref{eq:reformulation} it is easy to see that the series~$\tau^o_N$ is admissible.
It is also easy to see that $f^o_N=\sum_{i,j\ge 0}s^i s_-^jf^o_{N,i,j}$, where $f^o_{N,i,j}\in\mcR_N^{-i-2j}$. Thus, for any $j\ge 0$, the series $e^{\frac{s^3}{6}}\sum_{i\ge 0}s^if^o_{N,i,j}$ is
admissible. We conclude that the Fourier transform $\Phi^{form}_s\left[e^{\frac{s^3}{6}}f^o_N\right](\Lambda,z,s_-)$ with respect to the variable~$s$ is well-defined.

\begin{prop}\label{proposition:formal Fourier transform}
We have
$$
\Phi^{form}_s[\tau^o_N](\Lambda,z)=\Phi^{form}_s\left[e^{\frac{s^3}{6}}f^o_N\right]\left(\Lambda,z,\frac{z^2}{4}\right).
$$
\end{prop}
\begin{proof}
By Lemma~\ref{lemma:relation}, we have
$$
\tau^o_N=\left.e^{\frac{1}{2}\frac{\d^2}{\d s\d s_-}}\left(f^o_N e^{\frac{s^3}{6}}\right)\right|_{s_-=0}.
$$
From Lemma~\ref{lemma:Fourier of derivative} it follows that
\begin{multline*}
\Phi^{form}_s\left[e^{\frac{1}{2}\frac{\d^2}{\d s\d s_-}}\left(f^o_N e^{\frac{s^3}{6}}\right)\right](\Lambda,z,s_-)=e^{\frac{z^2}{4}\frac{\d}{\d s_-}}\left(\Phi^{form}_s\left[f^o_N
e^{\frac{s^3}{6}}\right](\Lambda,z,s_-)\right)=\\
=\Phi^{form}_s\left[f^o_N e^{\frac{s^3}{6}}\right]\left(\Lambda,z,s_-+\frac{z^2}{4}\right).
\end{multline*}
Setting $s_-=0$, we get the statement of the proposition.
\end{proof}

Propositions~\ref{proposition:formal Fourier transform} and~\ref{proposition:matrix integral} together with Section~\ref{subsection:convergent matrix integral} imply the following proposition.

\begin{prop}\label{proposition:convergent Fourier transform}
The series~$\Phi^{form}_s[\tau^o_N](\Lambda,z)$ is the asymptotic expansion of the integral
$$
c_{\Lambda,N}\sqrt{\frac{z}{2\pi}}\int_{\mbR}e^{\frac{i}{6}s^3-\frac{1}{2}s^2 z}ds\int_{\mcH_N} e^{\frac{i}{6}\tr
H^3-\frac{1}{2}\tr H^2\Lambda}\det\left(\frac{\Lambda+\sqrt{\Lambda^2-z^2}-iH+is-z}{\Lambda+\sqrt{\Lambda^2-z^2}-iH-is+z}\right)dH,
$$
when $0<z<\lambda_i$ and $z,\lambda_i\to+\infty$.
\end{prop}

%%%%%%%%%%%%%%%%%%%%%%%%%%%%%%%%%%%%%%%%%%%%%%%%%%%%%%%%%%%%%%%%%%%%%
%%%%%%%%%%%%%%%%%%%%%%%%%%%%%%%%%%%%%%%%%%%%%%%%%%%%%%%%%%%%%%%%%%%%%

\section{Analytical computations with the matrix integral}

Let~$N$ be an arbitrary positive integer. Recall that the series $\tau^c_N(\Lambda)\in\mbC[[\lambda_1^{-1},\ldots,\lambda_N^{-1}]]$ is defined by
$$
\tau^c_N(\Lambda):=\left.\tau^c\right|_{t_i=-(2i-1)!!\tr\Lambda^{-2i-1}}.
$$
Let $\Lambda_z:=\diag(\lambda_1,\ldots,\lambda_N,z)$. We set the genus parameter~$u$ to be equal to~$1.$ The following proposition is the key step in the proof of Theorem~\ref{theorem:main}.
\begin{prop}\label{proposition:main}
We have
\begin{gather}\label{eq:main equation}
\Phi^{form}_s\left[\tau^o_N\right]\left(\Lambda,z\right)=\tau^c_{N+1}(\Lambda_z)\sqrt{\det\left(\frac{1-z\Lambda^{-1}}{1+z\Lambda^{-1}}\right)}.
\end{gather}
\end{prop}
Before proving the proposition let us make a remark about the right-hand side of equation~\eqref{eq:main equation}. We see that
$\tau^c_{N+1}(\Lambda_z)\in\mbC[[\lambda_1^{-1},\ldots,\lambda_N^{-1}]][[z^{-1}]]$ and
$\sqrt{\det\left(\frac{1-z\Lambda^{-1}}{1+z\Lambda^{-1}}\right)}\in\mbC[[\lambda_1^{-1},\ldots,\lambda_N^{-1}]][[z]]$. So, on the right-hand side of~\eqref{eq:main equation} we multiply a power series in~$z^{-1}$ and a power series in~$z$. In general, the multiplication of two such series may not be well-defined. In our case, the issue
is resolved as follows. We have
$$
\sqrt{\det\left(\frac{1-z\Lambda^{-1}}{1+z\Lambda^{-1}}\right)}=e^{-\sum_{k\ge 0}\frac{1}{2k+1}z^{2k+1}\tr\Lambda^{-2k-1}}.
$$
Therefore, $\sqrt{\det\left(\frac{1-z\Lambda^{-1}}{1+z\Lambda^{-1}}\right)}$ has the form $\sum_{m\ge 0}f_m z^m$, where $f_m\in\mbC[\lambda_1^{-1},\ldots,\lambda_N^{-1}]$ is homogeneous of degree~$-m$. Thus, the product on the right-hand side of~\eqref{eq:main equation} is well-defined.

Our proof of Proposition~\ref{proposition:main} uses a famous technique, that is sometimes called the averaging procedure over the unitary group. We recall it in
Section~\ref{subsection:averaging}. After that, in Section~\ref{subsection:proof of the main proposition}, we prove Proposition~\ref{proposition:main}.

%%%%%%%%%%%%%%%%%%%%%%%%%%%%%%%%%%%%%%%%%%%%%%%%%%%%%%%%%%%%%%%%%%%%%%

\subsection{Averaging procedure over the unitary group}\label{subsection:averaging}

\subsubsection{Polar decomposition}

It is well-known that an arbitrary Hermitian matrix $H$ admits a polar decomposition, $H=U M U^{-1}$, where $U$ is a unitary matrix and $M$ is a diagonal matrix with real entries.
Denote by $\mcU_N$ the group of unitary $N\times N$ matrices. Given a vector $\overline m=(m_1,\ldots,m_N)\in\mbR^N$, define a map $\pi_{\overline m}\colon\mcU_N\to\mcH_N$ by
$$
\pi_{\overline m}(U):=U\diag(m_1,\ldots,m_N)U^{-1}.
$$
For a Hermitian matrix~$H,$ let $\overline m(H)=(m_1(H),\ldots,m_N(H))$ be the vector of its eigenvalues. It is defined up to a permutation of the coordinates. For
a subset~$D\subset\mbR^N$, that is invariant under permutations of the coordinates, let
$$
\mcH_N^D:=\{H\in\mcH_N|\overline m(H)\in D\}\subset\mcH_N.
$$
Suppose moreover that $D\subset\mbR^N$ is compact and measurable. Let $f\colon\mcH_N\to\mbC$ be an arbitrary smooth function. We denote by $dU$ the Haar probability measure on
$\mcU_N$. We have the following result (see e.g.~\cite[Sections~3 and~4]{LZ04}).

\begin{lemma}\label{lemma:polar}
We have
$$
\int_{\mcH^D_N}f(H)dH=\frac{\pi^{\frac{N^2-N}{2}}}{\prod_{i=1}^N i!}\int_{D}\prod_{1\le i<j\le N}(m_j-m_i)^2dm_1\ldots dm_N\int_{\mcU_N}(\pi_{\overline m}^*f)dU.
$$
\end{lemma}

\subsubsection{Harish-Chandra-Itzykson-Zuber formula}

Let $a_1,\ldots,a_N,b_1,\ldots,b_N$ be arbitrary real numbers such that $a_i\ne a_j$ and $b_i\ne b_j$, for $i\ne j$. Let $A:=\diag(a_1,\ldots,a_N)$ and $B:=\diag(b_1,\ldots,b_N)$. Let
$t$ be an arbitrary non-zero complex parameter. The Harish-Chandra-Itzykson-Zuber (HCIZ) formula says that
\begin{gather}\label{eq:HCIZ formula}
\int_{U_N}e^{t\tr(AUBU^{-1})}dU=\frac{\prod_{i=1}^{N-1}i!}{t^{\frac{N^2-N}{2}}}\frac{\det(e^{t a_i b_j})}{\prod_{1\le i<j\le N}(a_j-a_i)(b_j-b_i)}.
\end{gather}
The right-hand has no poles on the diagonals $\{a_i=a_j\}$ or $\{b_i=b_j\}$, and so it defines a smooth
function of $2N$ real parameters $a_1,\ldots,a_N,b_1,\ldots,b_N$. Formula~\eqref{eq:HCIZ formula} was originally found in~\cite{H-C57} and then was rediscovered in~\cite{IZ80}.

We will apply the HCIZ formula in the following way. Again, let $D$ be a compact measurable subset of $\mbR^N,$ invariant under permutations of the coordinates. Let $f\colon\mcH_N\to\mbC$ be a smooth unitary invariant function: $f(UHU^{-1})=f(H)$, for any $U\in\mcU_N$. Finally, consider pairwise distinct purely imaginary complex numbers $a_1,\ldots,a_N$ and set $A:=\diag(a_1,\ldots,a_N)$.
\begin{lemma}\label{lemma:averaging}
We have
\begin{gather}\label{eq:averaging}
\int_{\mcH_N^D}f(H)e^{\half\tr H^2 A}dH=\frac{(2\pi)^{\frac{N^2-N}{2}}}{N!\prod_{1\le i<j\le N}(a_j-a_i)}\int_D f(M)\det\left(e^{\half m_i^2 a_j}\right)\prod_{1\le i<j\le
N}\frac{m_j-m_i}{m_j+m_i}dM,
\end{gather}
where $M:=\diag(m_1,\ldots,m_N)$ and $dM:=dm_1\ldots dm_N$.
\end{lemma}
\begin{proof}
The lemma follows from combining Lemma~\ref{lemma:polar} and the HCIZ formula~\eqref{eq:HCIZ formula}.
\end{proof}

%%%%%%%%%%%%%%%%%%%%%%%%%%%%%%%%%%%%%%%%%%%%%%%%%%%%%%%%%%%%%%%%%%%%%

\subsection{Proof of Proposition~\ref{proposition:main}}\label{subsection:proof of the main proposition}

Suppose that $\lambda_1,\ldots,\lambda_N$ and $z$ are positive real numbers. Consider Kontsevich's integral
\begin{gather}\label{eq:Kontsevich's integral}
c_{\Lambda_z,N+1}\int_{\mcH_{N+1}}e^{\frac{i}{6}\tr H^3-\frac{1}{2}\tr H^2\Lambda_z}dH.
\end{gather}
As we already recalled in Section~\ref{subsection:convergent matrix integral}, the integral~\eqref{eq:Kontsevich's integral} is absolutely convergent and its asymptotic expansion, when $\lambda_i,z\to+\infty$, is given by the series $\tau^c_{N+1}(\Lambda_z)$ (see~\cite{Kon92}). Suppose moreover that $z<\lambda_i$. Using Proposition~\ref{proposition:convergent Fourier transform} we see that equation~\eqref{eq:main equation} is a consequence of the following equation:
\begin{multline}\label{eq:main convergent equation}
c_{\Lambda_z,N+1}\int_{\mcH_{N+1}}e^{\frac{i}{6}\tr H^3-\frac{1}{2}\tr H^2\Lambda_z}dH=\\
=c_{\Lambda,N}\sqrt{\det\left(\frac{\Lambda+z}{\Lambda-z}\right)}\sqrt{\frac{z}{2\pi}}\int_{\mbR}e^{\frac{i}{6}s^3-\frac{1}{2}s^2 z}ds\int_{\mcH_N} e^{\frac{i}{6}\tr
H^3-\frac{1}{2}\tr H^2\Lambda}\det\left(\frac{\Lambda+\sqrt{\Lambda^2-z^2}-iH+is-z}{\Lambda+\sqrt{\Lambda^2-z^2}-iH-is+z}\right)dH.
\end{multline}

The integral~\eqref{eq:Kontsevich's integral} is still absolutely convergent, when~$\lambda_i$'s and~$z$ are complex numbers with positive real parts. Moreover, the integral~\eqref{eq:Kontsevich's
integral} is semi-convergent, if $\lambda_i,z$ are purely imaginary non-zero complex numbers, and the asymptotic expansion of it, when $\lambda_i,z\to\infty$ is still given by the power series $\tau^c_{N+1}(\Lambda_z)$. See the discussion of these subtle questions in~\cite[page 208]{DIZ93}.

Let $\lambda_i$ and $z$ be purely imaginary complex numbers with positive imaginary parts. Assume moreover that $
|z|<|\lambda_i|$. Consider the diagonal matrix
$$
A:=\sqrt{\Lambda^2-z^2},
$$
where we choose a particular value of the square root, such that all entries of the matrix~$A$ have positive imaginary parts. Let $a_1,\ldots,a_N$ be the diagonal elements of~$A$. Let
$$
A_z:=\sqrt{\Lambda_z^2-z^2}=\diag(a_1,\ldots,a_N,0).
$$
Let us perform the change of variables $H\mapsto H-i(\Lambda_z+A_z)$ on the left-hand side of equation~\eqref{eq:main convergent equation}. It occurs that this kind of shifts is very useful in the study of Hermitian matrix models (see e.g.~\cite[page 5668]{IZ92} or~\cite[page 208]{DIZ93}). We get
$$
c_{\Lambda_z,N+1}e^{\tr B_z}\int_{\mcH_{N+1}}e^{\frac{i}{6}\tr H^3+\frac{1}{2}\tr H^2A_z+i\frac{z^2}{2}\tr H}dH,
$$
where $B_z:=\frac{(A_z+\Lambda_z)^2(2\Lambda_z-A_z)}{6}$. Define~$a_{N+1}$ to be equal to~$0$. By Lemma~\ref{lemma:averaging}, the last integral is equal to
\begin{gather}\label{exp:integral1}
\frac{(2\pi)^{\frac{(N+1)^2-N-1}{2}}c_{\Lambda_z,N+1}e^{\tr B_z}}{(N+1)!\prod_{1\le i<j\le N+1}(a_j-a_i)}\int_{\mbR^{N+1}}e^{\frac{i}{6}\tr M^3+i\frac{z^2}{2}\tr
M}\det\left(e^{\frac{1}{2}m_i^2a_j}\right)_{1\le i,j\le N+1}\prod_{1\le i<j\le N+1}\frac{m_j-m_i}{m_j+m_i}dM,
\end{gather}
where $M:=\diag(m_1,\ldots,m_{N+1})$ and $dM:=dm_1\ldots dm_{N+1}$. Note that
\begin{gather*}
c_{\Lambda_z,N+1}=c_{\Lambda,N}(2\pi)^{-N}\sqrt\frac{z}{2\pi}\prod_{i=1}^N(z+\lambda_i)\quad\text{and}\quad\tr B_z=\tr B+\frac{z^3}{3},
\end{gather*}
where $B=\frac{(A+\Lambda)^2(2\Lambda-A)}{6}$. Expanding the determinant along the last column, we get
\begin{align*}
&\det\left(e^{\frac{1}{2}m_i^2a_j}\right)_{1\le i,j\le N+1}\prod_{1\le i<j\le N+1}\frac{m_j-m_i}{m_j+m_i}=\\
&\hspace{2cm}=\sum_{k=1}^{N+1}(-1)^{N+1-k}\det\left(e^{\frac{1}{2}m_i^2a_j}\right)_{\substack{1\le i\le N+1,\,i\ne k\\1\le j\le N}}\prod_{1\le i<j\le N+1}\frac{m_j-m_i}{m_j+m_i}=\\
&\hspace{2cm}=\sum_{k=1}^{N+1}\det\left(e^{\frac{1}{2}m_i^2a_j}\right)_{\substack{1\le i\le N+1,\,i\ne k\\1\le j\le N}}\prod_{\substack{1\le i<j\le N+1\\i,j\ne k}}\frac{m_j-m_i}{m_j+m_i}\prod_{\substack{1\le i\le N+1\\i\ne k}}\frac{m_k-m_i}{m_k+m_i}.
\end{align*}
When we substitute this sum in expression~\eqref{exp:integral1}, we see that all the $N+1$ summands give the same contribution to the integral. Therefore, we can rewrite~\eqref{exp:integral1} as follows:
\begin{align}
&\sqrt{\frac{z}{2\pi}}\prod_{i=1}^N\frac{z+\lambda_i}{a_i}e^{\frac{z^3}{3}}\int_{\mbR}e^{\frac{i}{6}m_{N+1}^3+\frac{i}{2}m_{N+1}z^2}dm_{N+1}\times\label{exp:integral2}\\
&\times\left[\frac{(2\pi)^{\frac{N^2-N}{2}}c_{\Lambda,N}e^{\tr B}}{N!\prod_{1\le i<j\le N}(a_j-a_i)}\int_{\mbR^{N}}e^{\frac{i}{6}\tr (M')^3+i\frac{z^2}{2}\tr
M'}\det\left(e^{\frac{1}{2}m_i^2 a_j}\right)_{1\le i,j\le N}\times\right.\notag\\
&\left.\times\prod_{1\le i<j\le N}\frac{m_j-m_i}{m_j+m_i}\det\left(\frac{M'-m_{N+1}}{M'+m_{N+1}}\right)dM'\right],\notag
\end{align}
where $M':=\diag(m_1,\ldots,m_N)$ and $dM':=dm_1\ldots dm_N$. By Lemma~\ref{lemma:averaging}, the expression in the square brackets is equal to
$$
c_{\Lambda,N}e^{\tr B}\int_{\mcH_N}e^{\frac{i}{6}\tr H^3+i\frac{z^2}{2}\tr H+\frac{1}{2}\tr H^2 A}\det\frac{H-m_{N+1}}{H+m_{N+1}}dH.
$$
Redenoting $m_{N+1}$ by $s$, expression~\eqref{exp:integral2} is equal to
\begin{gather}\label{exp:integral3}
\sqrt{\frac{z}{2\pi}}\det\left(\frac{z+\Lambda}{A}\right)e^{\frac{z^3}{3}}\int_{\mbR}e^{\frac{i}{6}s^3+\frac{i}{2}sz^2}ds\left[c_{\Lambda,N}e^{\tr B}\int_{\mcH_N}e^{\frac{i}{6}\tr
H^3+\frac{1}{2}\tr H^2 A+i\frac{z^2}{2}\tr H}\det\frac{H-s}{H+s}dH\right].
\end{gather}
Now we make the shift $H\mapsto H+i(\Lambda+A)$ in the integral in the square brackets. Using also that $A=\sqrt{\Lambda^2-z^2}=\sqrt{(\Lambda-z)(\Lambda+z)}$, we get that~\eqref{exp:integral3} is equal to
\begin{gather*}
\sqrt{\frac{z}{2\pi}}\sqrt{\det\left(\frac{\Lambda+z}{\Lambda-z}\right)}e^{\frac{z^3}{3}}\int_{\mbR}e^{\frac{i}{6}s^3+\frac{i}{2}sz^2}ds\left[c_{\Lambda,N}\int_{\mcH_N}e^{\frac{i}{6}\tr
H^3-\frac{1}{2}\tr H^2\Lambda}\det\left(\frac{\Lambda+A-iH+is}{\Lambda+A-iH-is}\right)dH\right].
\end{gather*}
Finally, making the shift $s\mapsto s+iz$, we come to
\begin{gather*}
c_{\Lambda,N}\sqrt{\det\left(\frac{\Lambda+z}{\Lambda-z}\right)}\sqrt{\frac{z}{2\pi}}\int_{\mbR}e^{\frac{i}{6}s^3-\frac{1}{2}s^2 z}ds\int_{\mcH_N} e^{\frac{i}{6}\tr
H^3-\frac{1}{2}\tr H^2\Lambda}\det\left(\frac{\Lambda+A-iH+is-z}{\Lambda+A-iH-is+z}\right)dH.
\end{gather*}
Equation~\eqref{eq:main convergent equation} is proved, and, thus, also Proposition~\ref{proposition:main}.

%%%%%%%%%%%%%%%%%%%%%%%%%%%%%%%%%%%%%%%%%%%%%%%%%%%%%%%%%%%%%%%%%%%%
%%%%%%%%%%%%%%%%%%%%%%%%%%%%%%%%%%%%%%%%%%%%%%%%%%%%%%%%%%%%%%%%%%%%

\section{Proof of Theorem~\ref{theorem:main}}

From the dimension constraint~\eqref{open dimension} it follows that, if an open intersection number~\eqref{openproducts} is non-zero, then the genus~$g$ is uniquely determined by~$a_i$'s and~$k$. Using this observation, it is easy to show that, if the open KdV equations hold for~$u=1$, then they are also true for an arbitrary~$u$. The same is true for the open Virasoro equations. Therefore, without loss of generality, we can assume that~$u=1$.

Let~$G_z$ be the shift operator which acts on a series $f(t_0,t_1,\ldots)\in \mbC[[t_0,t_1,\ldots]]$ by
$$
G_z[f](t_0,t_1,\ldots):=f\left(t_0-\frac{k_0}{z},t_1-\frac{k_1}{z^3},t_2-\frac{k_2}{z^5},\ldots\right),
$$
where $k_n:=(2n-1)!!$ and, by definition, $(-1)!!:=1$. Let $\xi^c:=\sum_{i\ge 0}\frac{t_iz^{2i+1}}{(2i+1)!!}$.

Proposition~\ref{proposition:main} implies the following corollary.
\begin{cor}\label{corollary:main}
We have
\begin{gather}\label{eq:eq1}
\Phi^{form}_s\left[\tau^o\right]=G_z[\tau^c]e^{\xi^c}.
\end{gather}
\end{cor}
\begin{proof}
It is easy to check that for an arbitrary $N\ge 1$ we have
$$
\left.G_z[\tau^c]\right|_{t_i=-(2i-1)!!\tr\Lambda^{-2i-1}}=\tau^c_{N+1}(\Lambda_z)\quad\text{and}\quad \left.e^{\xi^c}\right|_{t_i=-(2i-1)!!\tr\Lambda^{-2i-1}}=\sqrt{\det\left(\frac{1-z\Lambda^{-1}}{1+z\Lambda^{-1}}.\right)}.
$$
Therefore, by Proposition~\ref{proposition:main} and equation~\eqref{eq:relation of two transforms}, we have
\begin{gather}\label{eq:after substitution}
\left.\left(\Phi^{form}_s\left[\tau^o\right]\right)\right|_{t_i=-(2i-1)!!\tr\Lambda^{-2i-1}}=\left.\left(G_z[\tau^c]e^{\xi^c}\right)\right|_{t_i=-(2i-1)!!\tr\Lambda^{-2i-1}}.
\end{gather}
For an arbitrary $m\ge 0$ the functions $\tr\Lambda^{-1},\tr\Lambda^{-3},\ldots,\tr\Lambda^{-2m-1}$ become algebraically independent, when~$N$ is sufficiently large. Therefore, equation~\eqref{eq:after substitution} implies that $\Phi^{form}_s\left[\tau^o\right]=G_z[\tau^c]e^{\xi^c}$. The corollary is proved.
\end{proof}

Let us prove that the open free energy~$F^o$ satisfies the half of the Burgers-KdV hierarchy \eqref{eq:t-flows}-\eqref{eq:s-flow}. In~\cite{Bur14b} it was obtained that the half of the Burgers-KdV hierarchy can be written in a very convenient form using the Lax formalism. Let us briefly recall this result. A pseudo-differential operator $A$ is a Laurent series
$$
A=\sum_{n=-\infty}^m a_n(t)\d_x^n,
$$
where $m$ is an arbitrary integer and $a_n$ are formal power series in $t_0,t_1,\ldots$. We will always identify~$t_0$ with~$x$. Let
\begin{gather*}
A_+:=\sum_{n=0}^m a_n\d_x^n.
\end{gather*}
The product of pseudo-differential operators is defined by the following commutation rule:
\begin{gather*}
\d_x^k\circ f:=\sum_{l=0}^\infty\frac{k(k-1)\ldots(k-l+1)}{l!}\frac{\d^l f}{\d x^l}\d_x^{k-l},
\end{gather*}
where $k\in\Z$ and $f\in\mbC[[t_0,t_1,\ldots]]$. For any $m\ge 2$ and a pseudo-differential operator~$A$ of the form
$$
A=\d_x^m+\sum_{n=1}^\infty a_n\d_x^{m-n},
$$
there exists a unique pseudo-differential operator $A^{\frac{1}{m}}$ of the form
$$
A^{\frac{1}{m}}=\d_x+\sum_{n=0}^\infty \widetilde{a}_n\d_x^{-n},
$$
such that $\left(A^{\frac{1}{m}}\right)^m=A$.

Consider the pseudo-differential operator $L:=\d_x^2+2F^c_{t_0,t_0}$. In~\cite{Bur14b} it was proved that the half of the Burgers-KdV system~\eqref{eq:t-flows}-\eqref{eq:s-flow} is equivalent to the following system:
\begin{align}
\frac{\d}{\d t_n}e^{F}=&\frac{1}{(2n+1)!!}\left(L^{n+\half}\right)_+e^{F}\label{eq:first},\\
\frac{\d}{\d s}e^{F}=&\frac{1}{2}Le^{F}.\label{eq:second}
\end{align}
Let us prove that the open free energy~$F^o$ satisfies equations~\eqref{eq:first} and~\eqref{eq:second}. Let
\begin{gather*}
\xi(t_*,s_*,z):=\xi^c+\sum_{n\ge 0}\frac{s_nz^{2n+2}}{(2n+2)!!},\quad\text{and}\quad \psi(t_*,s_*,z):=\frac{G_z(\tau^c)}{\tau^c}e^{\xi}.
\end{gather*}
From the fact that the closed partition function~$\tau^c$ is a tau-function of the KdV hierarchy it follows that the series $\psi$ is the wave function of the KdV hierarchy (see e.g.~\cite{Dic03}), that is, it satisfies the following equations:
\begin{align}
\frac{\d}{\d t_n}\psi=&\frac{1}{(2n+1)!!}\left(L^{n+\half}\right)_+\psi,\label{eq:first for the wave}\\
\frac{\d}{\d s_n}\psi=&\frac{1}{2^{n+1}(n+1)!}L^{n+1}\psi\label{eq:second for the wave}.
\end{align}
From Corollary~\ref{corollary:main} it follows that
\begin{gather}\label{eq:eq2}
\Phi^{form}_s\left[e^{F^o}\right]=\psi|_{s_*=0}.
\end{gather}

Let us prove~\eqref{eq:first}. Let $O_{t_n}:=\frac{\d}{\d t_n}-\frac{1}{(2n+1)!!}\left(L^{n+\half}\right)_+$. We have
$$
\Phi^{form}_s\left[O_{t_n}e^{F^o}\right]=O_{t_n}\Phi^{form}_s\left[e^{F^o}\right]\stackrel{\text{by~\eqref{eq:eq2}}}{=}O_{t_n}\left(\psi|_{s_*=0}\right)=\left(O_{t_n}\psi\right)|_{s_*=0}\stackrel{\text{by~\eqref{eq:first
for the wave}}}{=}0.
$$
By Lemma~\ref{lemma:injectivity}, we get $O_{t_n}e^{F^o}=0$ and equation~\eqref{eq:first} is proved.

Let us prove~\eqref{eq:second}. Let $O_s:=\frac{\d}{\d s}-\frac{1}{2}L$. We have
$$
\Phi^{form}_s\left[O_s e^{F^o}\right]\stackrel{\text{by Lemma~\ref{lemma:Fourier of
derivative}}}{=}\left(\frac{z^2}{2}-\frac{1}{2}L\right)\Phi^{form}_s\left[e^{F^o}\right]\stackrel{\text{by~\eqref{eq:eq2}}}{=}\left(\frac{z^2}{2}-\frac{1}{2}L\right)\left(\psi|_{s_*=0}\right)=\left(O_s\psi\right)|_{s_*=0}\stackrel{\text{by~\eqref{eq:second
for the wave}}}{=}0.
$$
By Lemma~\ref{lemma:injectivity}, we get $O_s e^{F^o}=0$ and equation~\eqref{eq:second} is proved.

In~\cite{Bur14a} it was shown that the open KdV equations follow from the half of the Burgers-KdV hierarchy. Therefore, Conjecture~\ref{conjecture:open KdV} is proved. We also see that the open free energy~$F^o$ satisfies the initial condition~$F^o|_{t_{\ge 1}=0,s=0}=0$. In~\cite{Bur14a} it was proved that such a solution of the half of the Burgers-KdV hierarchy satisfies the open Virasoro equations~\eqref{eq:open virasoro}. Thus, Conjecture~\ref{conjecture:open Virasoro} is also proved.

%%%%%%%%%%%%%%%%%%%%%%%%%%%%%%%%%%%%%%%%%%%%%%%%%%%%%%%%%%%%%%%%%%
%%%%%%%%%%%%%%%%%%%%%%%%%%%%%%%%%%%%%%%%%%%%%%%%%%%%%%%%%%%%%%%%%%

\section{Virasoro equations}

The derivation of the open Virasoro equations~\eqref{eq:open virasoro} from the open KdV equations~\eqref{eq:openkdv} was given in~\cite{Bur14a}. Another derivation was obtained in~\cite{Bur14b}. In this section we want to show that the open Virasoro equations can be very easily obtained from Corollary~\ref{corollary:main}. Again we can assume that~$u=1$.

By Corollary~\ref{corollary:main},
$$
\Phi_s^{form}[\tau^o]=G_z[\tau^c]e^{\xi^c}.
$$
From the closed Virasoro equations~\eqref{vira} it is easy to derive that (see e.g.~\cite{Bur14b})
$$
L_n\left(G_z[\tau^c]e^{\xi^c}\right)=\left(\frac{z^{2n+2}}{2^{n+1}}\left(\frac{1}{z}\frac{\d}{\d
z}-\frac{1}{2z^2}-z\right)+\frac{n+1}{2^{n+2}}z^{2n}\right)\left(G_z[\tau^c]e^{\xi^c}\right).
$$
By Lemmas~\ref{lemma:Fourier of derivative} and~\ref{lemma:Fourier and multiplication}, we have
\begin{align*}
\left(\frac{z^{2n+2}}{2^{n+1}}\left(\frac{1}{z}\frac{\d}{\d
z}-\frac{1}{2z^2}-z\right)+\frac{n+1}{2^{n+2}}z^{2n}\right)\Phi_s^{form}[\tau^o]&=\Phi_s^{form}\left[\left(-\frac{\d^{n+1}}{\d s^{n+1}}\circ s+\frac{n+1}{4}\frac{\d^n}{\d
s^n}\right)\tau^o\right]=\\
&\quad\quad=-\Phi_s^{form}\left[\left(s\frac{\d^{n+1}}{\d s^{n+1}}+\frac{3n+3}{4}\frac{\d^n}{\d s^n}\right)\tau^o\right].
\end{align*}
Here, as above, the circle~$\circ$ means the composition of operators. In the last operator inside the Fourier transform we immediately recognize the difference between the open Virasoro operator~$\mcL_n$ and the closed Virasoro operator~$L_n$ (see~\eqref{eq:open Virasoro operator}). Therefore, we get
$$
\Phi_s^{form}[\mcL_n\tau^o]=0.
$$
Lemma~\ref{lemma:injectivity} implies that $\mcL_n\tau^o=0$. The open Virasoro equations are proved.

\bibliographystyle{amsabbrvc}
%\bibliography{bibli}

\end{document}